\documentclass[12pt, reqno]{amsart}
\usepackage[all]{xy}
\usepackage{amssymb}
\usepackage{amsthm}
\usepackage{hyperref}
\usepackage{todonotes}
\hypersetup{colorlinks=true,linkcolor=blue,citecolor=magenta}
\usepackage{amsmath}
\usepackage{mathtools}
\usepackage{amscd,enumitem}
\usepackage{verbatim}
\usepackage{eurosym}
\usepackage{float}
\usepackage{color}
\usepackage{cases}
\usepackage{bm}
\usepackage{dcolumn}
\usepackage[mathscr]{eucal}
\usepackage[all]{xy}
\usepackage{hyperref}
\usepackage{bbm}
\usepackage{comment}
\usepackage{multicol}
\usepackage{adjustbox}
\usepackage[textheight=8.77in, textwidth=6.8in]{geometry}

\usepackage{constants}

\usepackage[]{breqn}

\newtheorem*{thm*}{Theorem}
\newtheorem*{conj*}{Conjecture}

\newtheorem{theorem}{Theorem}[section]

\newtheorem{lemma}[theorem]{Lemma}

\newtheorem{conj}[theorem]{Conjecture}
\newtheorem{remark}[theorem]{Remark}

\newtheorem{corollary}[theorem]{Corollary}

\newcommand{\Z}{\mathbb{Z}}
\newcommand{\Q}{\mathbb{Q}}

\newcommand{\R}{\mathbb{R}}

 \newcommand{\Sym}{\mathrm{Sym}}
\DeclareMathOperator{\Li}{Li}

\DeclareMathOperator{\ST}{ST}
\DeclareMathOperator{\Res}{Res}
\DeclareMathOperator{\Real}{Re}
\DeclareMathOperator{\Imag}{Im}
\renewcommand{\Re}{\Real}
\renewcommand{\Im}{\Imag}

\newconstantfamily{abcon}{symbol=c}

\makeatletter
\DeclarePairedDelimiterX{\pmodx}[1]{(}{)}{{\operator@font mod}\mkern6mu#1}
\renewcommand{\pmod}{%
  \allowbreak
  \if@display\mkern18mu\else\mkern8mu\fi
  \pmodx
}
\makeatother

\numberwithin{equation}{section}

\begin{document}
\title[An unconditional explicit bound on the error in Sato--Tate]{An Unconditional Explicit Bound on the Error Term in the Sato--Tate Conjecture}

\author[A. Hoey, J. Iskander, S. Jin, F. Trejos Su\'arez]{Alexandra Hoey, Jonas Iskander, Steven Jin, Fernando Trejos Su\'arez}
\address{Department of Mathematics, Massachusetts Institute of Technology, Cambridge, MA 02139}
\email[A. Hoey]{ahoey@mit.edu}
\address{Department of Mathematics, Harvard University, Cambridge, MA 02138}
\email[J. Iskander]{jonasiskander@college.harvard.edu}
\address{Department of Mathematics, University of Maryland, College Park, MD 20742}
\email[S. Jin]{sjin6816@umd.edu}
\address{Department of Mathematics, Yale University, New Haven, CT 06511}
\email[F. Trejos Su\'arez]{fernando.trejos@yale.edu}

\begin{abstract}
Let $f(z) = \sum_{n=1}^\infty a_f(n)q^n$ be a holomorphic cuspidal newform with even integral weight $k\geq 2$, level $N$, trivial nebentypus, and no complex multiplication (CM). For all primes $p$, we may define $\theta_p\in [0,\pi]$ such that $a_f(p) = 2p^{(k-1)/2}\cos \theta_p$. The Sato--Tate conjecture states that the angles $\theta_p$ are equidistributed with respect to the probability measure $\mu_{\textrm{ST}}(I) = \frac{2}{\pi}\int_I \sin^2 \theta \; d\theta$, where $I\subseteq [0,\pi]$. Using recent results on the automorphy of symmetric power $L$-functions due to Newton and Thorne, we explicitly bound the error term in the Sato--Tate conjecture when $f$ corresponds to an elliptic curve over $\Q$ of arbitrary conductor or when $f$ has squarefree level. In these cases, if $\pi_{f,I}(x) := \#\{ p \leq x : p \nmid N, \theta_p\in I\}$, and $\pi(x) := \# \{ p \leq x \}$, we prove the following bound:
\begin{equation*}
     \left| \frac{\pi_{f,I}(x)}{\pi(x)} - \mu_{\textrm{ST}}(I)\right| \leq 58.1\frac{\log((k-1)N \log{x})}{\sqrt{\log{x}}} \qquad \text{for} \quad x \geq 3.
\end{equation*}
 As an application, we give an explicit bound for the number of primes up to $x$ that violate the Atkin--Serre conjecture for $f$.
\end{abstract}
\maketitle

\section{Introduction and statement of results} 

Let $f(z) = \sum_{n=1}^\infty a_f(n) e^{2\pi inz} \in S_{k}^{\text{new}}(\Gamma_0(N))$ be a non-CM holomorphic cuspidal newform with trivial nebentypus, level $N$, and even integral weight $k \geq 2$. Deligne's proof of the Weil conjectures implies the Weil--Deligne bound $|a_f(p)|\leq 2p^{(k-1)/2}$ at all primes $p$ and that $a_f(n)\in \mathbb{R}$ for all $n$. Consequently, we may define $\theta_p \in[0,\pi]$ such that $a_f(p) = 2p^{(k-1)/2}\cos \theta_p$. It is natural to ask how the angles $\theta_p$ are distributed. In particular, given an interval $I \subseteq [0,\pi]$, we wish to understand the behavior of $\pi_{f,I}(x) := \# \{p \leq x : p \nmid N,\ \theta_p \in I \}$. The Sato--Tate conjecture, now a theorem due to Barnet--Lamb, Geraghty, Harris, and Taylor \cite{BGHT}, asserts that
\begin{equation}
    \pi_{f,I}(x) \sim \mu_{\ST}(I)\pi(x), 
\end{equation}
where $\mu_{\ST}(I)$ is the probability measure given by $\frac{2}{\pi }\int_I \sin ^2 \theta\, d\theta$, and $\pi(x)$ is the ordinary prime-counting function.

Despite the successful proof, unconditional effective error bounds remained unattainable without knowing that the symmetric power $L$-functions $L(s,\Sym^m f)$ have particular analytic properties for any $m \geq 1$ (for instance, that they have an analytic continuation to $\mathbb{C}$ and satisfy a functional equation). By Langlands functoriality, these analytic properties would be immediately implied by the fact that the $m$-th symmetric power lift $\Sym^m \pi_f$ corresponds to a cuspidal automorphic representation of $GL_{m+1}(\mathbb{A}_{\mathbb{Q}})$, where $\pi_f$ is the automorphic representation of $GL_2(\mathbb{A}_{\mathbb{Q}})$ corresponding to $f$. Though it was known that any symmetric power $L$-function becomes automorphic after base change to a suitable number field, the result of \cite{BGHT} could not be made effective without first making this base change constructive; the proof of automorphy would remove the need for a base change altogether.

Let $N$ be squarefree. Assuming the automorphy of symmetric power $L$-functions over $\mathbb{Q}$ and the Generalized Riemann Hypothesis, Rouse and Thorner \cite{RT} proved the explicit error bound
\begin{equation*}
    |\pi_{f,I}(x)-\mu_{\ST}(I)\mathrm{Li}(x)| \leq 3.33x^{\frac{3}{4}}- \frac{3x^{\frac{3}{4}}\log \log x}{\log x}+\frac{202x^{\frac{3}{4}}\log((k-1)N)}{\log x}\qquad \text{for} \quad x \geq 2.
\end{equation*}
The automorphy of the symmetric power $L$-functions $L(s,\Sym^m f)$ over $\mathbb{Q}$ had long been expected for all $m$, but was known until recently only for $m\leq 8$ \cite{CT, GH, K, KS}. In 2019 and 2020, the result for all $m$ arrived in a pair of breakthrough papers of Newton and Thorne \cite{NT, NT2}. These made unconditional a result due to Thorner \cite{T2}, which states that for fixed $f$ and $I$ and for any $\varepsilon > 0$, there exist effectively computable constants $c_{1,\varepsilon}, c_{2,\varepsilon} > 0$ depending on $f$ such that
\begin{align*}
|\pi_{f,I}(x)-\mu_{\ST}(I)\pi(x)|\leq c_{1,\varepsilon} \pi(x)(\log x)^{-\frac{1}{8}+\varepsilon}\qquad \text{for} \quad x \geq c_{2,\varepsilon}.
\end{align*} In 2021, Thorner \cite{T} showed that there exists an effectively computable absolute constant $\Cr{c_{100}}$ such that
\begin{align*}
    \left|\pi_{f,I}(x) - \mu_{\ST}(I)\pi(x)\right|\leq \Cr{c_{100}}\pi(x) \frac{\log((k-1)N \log{x})}{\sqrt{\log{x}}} \qquad \text{for} \quad x \geq 3.
\end{align*}

In this paper, we make this constant $\Cr{c_{100}}$ explicit. In particular, we prove the following theorems, which constitute the first unconditional explicit bounds on the error term in the Sato--Tate conjecture.
\begin{theorem}\label{main2}
Let $f(z) = \sum_{n=1}^\infty a_f(n) q^n\in S_{k}^{\text{new}}(\Gamma_0(N))$ be a non-CM holomorphic cuspidal newform with trivial nebentypus, squarefree level $N$, and even integral weight $k \geq 2$. Then for $x\geq 3$, we have
\begin{equation*}
    \left| \pi_{f,I}(x) - \mu_{\ST}(I)\pi(x)\right| \leq 58.1\; \pi(x)\frac{\log((k-1) N \log{x})}{\sqrt{\log{x}}}.
\end{equation*}
\end{theorem}

\begin{theorem}\label{mainEC}
Let $f$ be the newform corresponding to a non-CM elliptic curve over $\mathbb{Q}$ with arbitrary conductor $N$. Then for $x\geq 3$, the same bound as in Theorem \ref{main2} applies.
\end{theorem}

\begin{remark}
The requirement that $N$ be squarefree in Theorem \ref{main2} is a technical inconvenience that can be avoided if one is willing to accept a weaker bound. This is discussed in Remark \ref{squarefree-remark}.
\end{remark}

\begin{remark}
The true order of magnitude of this error remains an open question. It is expected that a bound of the shape $c_{f,\varepsilon} x^{\frac{1}{2}+\varepsilon}$ is satisfied. Some convincing evidence for this is exhibited in \cite{BP} in the elliptic curve case.
\end{remark}

For $k=2$ and $f$ corresponding to a non-CM elliptic curve over $\mathbb{Q}$, a result of Elkies \cite{E} yields that $a_f(p)=0$ for infinitely many $p$. On the other hand, for $k\geq 4$, we expect that $a_f(p)$ takes on each real value only finitely many times. The Atkin--Serre conjecture \cite{S}, stated below, makes this precise.

\begin{conj}[Atkin--Serre]
 Let $f\in S_{k}^{\text{new}}(\Gamma_0(N))$ be a non-CM cuspidal newform of weight $k\geq 4$. Then for each $\varepsilon>0$, there is a $c_{\varepsilon,f}>0$ such that for sufficiently large $p$, we have 
\begin{align*}
    |a_f(p)| \geq c_{\varepsilon,f}\, p^{\frac{k-3}{2}-{\varepsilon}}.
\end{align*}
\end{conj}

\noindent
Following an argument similar to that of Gafni, Thorner, and Wong \cite{GTW}, we apply Theorem \ref{main2} to obtain the following result, which for $k \geq 4$ implies an upper bound on the number of primes up to $x$ that violate the Atkin--Serre conjecture. 
\begin{theorem}\label{AS}
  Let $f \in S_k^{\text{new}}(\Gamma_0(N))$ be a non-CM cuspidal newform of trivial nebentypus, squarefree level $N$, and even integral weight $k \geq 2$. Then we have
\begin{equation*}
      \frac{\# \left\{  x < p \leq 2x : |a_f (p)| \leq 2p^{(k-1)/2} \frac{\log \log p}{\sqrt{\log p}}\right\}}{\# \{ x<p \leq 2x\}} \leq  179 \cdot \frac{\log ((k-1)N \log{x})}{\sqrt{\log{x}}} . 
\end{equation*}
The same bound applies if $f$ is a newform corresponding to a non-CM elliptic curve over $\Q$ with arbitrary conductor $N$.  
\end{theorem} \noindent The work of Gafni et al.\@ implies that a density one subset of primes satisfies the Atkin--Serre conjecture for arbitrary $N$; Theorem \ref{AS} provides an explicit bound on the number of exceptional primes when $N$ is squarefree.

In our proof of Theorem \ref{main2}, we derive an analogue (Theorem \ref{main}) of the prime number theorem by proving explicit results on the horizontal and vertical distributions of nontrivial zeros of symmetric power $L$-functions. Explicit estimates for the ordinary prime counting function due to Dusart \cite{Du} and Trudgian \cite{Tr2} will then allow us to conclude Theorem \ref{main2} via partial summation.

While our work closely follows that of Thorner \cite{T}, who first established the shape of the asymptotic in our error term in the unconditional case, we maintain higher precision in our calculations to obtain an explicit result.\footnote{For our purposes, we do not require an explicit form of the log-free zero density estimate in \cite{T}.} 
The nature of our calculations is most similar in spirit to the work of Rouse and Thorner \cite{RT}, although working in the unconditional case (i.e., without GRH) requires us to compute an explicit zero-free region and complicates our treatment of the vertical distribution of zeros of $L(s, \Sym^m f)$. Moreover, we employ short-interval smoothing to improve our numerical result, as seen in our choice to integrate over the functions $\Theta_m(x)$ in Lemma \ref{erdos-turan}.

The structure of this paper is as follows. In Section \ref{section-symmetric-powers}, we review the necessary background on symmetric power $L$-functions. In Section \ref{section-main-theorem}, we prove Theorems \ref{main2} and \ref{mainEC} using a series of lemmas that will be established in Sections \ref{horizontalsection}-\ref{Section:PerronIntegral}. In Section \ref{horizontalsection}, we give an explicit zero-free region for $L(s, \Sym^m f)$. In Section \ref{verticalsection}, we give an upper bound for the number of nontrivial zeros of $L(s, \Sym^m f)$ up to height $T$. In Section \ref{Section:PerronIntegral}, we apply the results of Sections \ref{horizontalsection} and \ref{verticalsection} to prove the bounds used in Section \ref{section-main-theorem}. Finally, in Section \ref{section-atkin-serre}, we apply our main results to prove Theorem \ref{AS}.

\section*{Funding}

 This work was supported by the National Science Foundation [DMS 2002265, DMS 205118]; National Security Agency [H98230-21-1-0059]; the Thomas Jefferson Fund at the University of Virginia; and the Templeton World Charity Foundation.

\section*{Acknowledgements}

The authors would like to thank Jesse Thorner for advising this project and for many helpful conversations and suggestions, as well as Ken Ono for his valuable comments. We are grateful to Chantal David, Timothy Trudgian, Larry Washington, and Mark Watkins for several useful conversations, especially regarding bounding the conductors of symmetric power $L$-functions. We also thank the anonymous referee who reviewed our paper for helpful feedback and corrections. This project was completed as a part of the 2021 Number Theory REU at the University of Virginia. The authors used Wolfram Mathematica to simplify many computations.

\section{Background on Symmetric Power \texorpdfstring{$L$}{L}-Functions} \label{section-symmetric-powers}

For the duration of this section, let $f \in S_k^{\text{new}} (\Gamma_0(N))$ be a non-CM holomorphic cuspidal newform with level $N$, even integral weight $k\geq 2$, and trivial nebentypus. For each $m \geq 0$, recall that there exists the \textit{$m$-th symmetric power $L$-function associated to $f$}, denoted by
\begin{equation} \label{eq:sym_power_L_Euler_prod}
    L(s, \Sym^m f)= \prod_{p} \prod_{j = 0}^m \left( 1 - \frac{\alpha_{j, \Sym^m f}(p)}{p^s} \right)^{-1} \!\!\!\! := \sum_{n=1}^\infty \frac{a_{\Sym^m f}(n)}{n^s} \qquad \text{for} \quad \Re(s) > 1
\end{equation} where the $\alpha_{j,\Sym^m f}(p)$ are particular complex numbers which for $p \nmid N$ satisfy $\alpha_{j,\Sym^m f}(p) = e^{i(m-2j)\theta_p}$.  It follows that for $p \nmid N$, we can write $a_{\Sym^m f}(p) = U_m (\cos\theta_p)$, where $U_m(x)$ is the $m$-th Chebyshev polynomial of the second kind. For $p\mid N$, the values of the Satake parameters $\alpha_{j, \Sym^m f}(p)$ can be determined using \cite[Appendix]{ST}. Simple, explicit descriptions of $\alpha_{j, \Sym^m f}(p)$ which are uniform in $f$ are available when $N$ is square-free \cite{RT} as well as when $f$ corresponds to a non-CM elliptic curve via modularity \cite[Appendix]{David}, but these will not be used in our proofs. We note that from the definition of $L(s, \Sym^m f)$, it follows that $L(s, \Sym^0 f) = \zeta(s)$ and $L(s, \Sym^1 f)=L(s, f)$, the $L$-function of the newform $f$. For notational convenience, it is also occasionally useful to define $L(s, \Sym^{-1} f) := 1$.
 
Recent results due to Newton and Thorne (\cite[Theorem B]{NT} and \cite[Theorem A]{NT2}) imply that $\Sym^m f$ corresponds to a cuspidal automorphic representation of $\mathrm{GL}_{m+1}(\mathbb{A}_{\mathbb{Q}})$ for all $m \geq 1$. Indeed, $\Sym^m f$ is a unitary cuspidal automorphic representation that is regular, algebraic, and self-dual and hence can be realized in the cohomology of certain Shimura varieties as per the results of Harris and Taylor \cite{HT}. In light of this result, each Satake parameter $\alpha_{j, \Sym^m f}(p)$ is bounded in modulus by $1$, as is remarked in \cite{CJ}. 

In \cite[Theorem 6.1]{T}, Thorner applies Newton and Thorne's results, together with previous conditional results due to Moreno and Shahidi \cite{MS} and Cogdell and Michel \cite[Section 3]{CM}, to deduce several important analytic properties of $L(s, \Sym^m f)$ for $m \geq 1$, which may be stated as follows:
\begin{itemize}
    \item[(i)] The gamma factor\footnote{The expression we give corrects for an extra factor of the arithmetic conductor of $L(s, \Sym^m f)$ in the gamma factor stated in \cite[6.1]{T}.} of $L(s, \Sym^m f)$ corresponding to the infinite place of $\mathbb{Q}$ is given by
    \begin{equation}\label{gammasconj11}
            \gamma(s, \Sym^m f) = \begin{cases} \displaystyle \prod_{j=1}^{\frac{m+1}{2}} \Gamma_{\mathbb{C}} (s + (j-{\textstyle\frac{1}{2})(k-1)}) & \text{if $m$ is odd},\\
            \displaystyle \Gamma_{\mathbb{R}} (s + r) \prod_{j=1}^{\frac{m}{2}} \Gamma_{\mathbb{C}} (s + j(k-1)) & \text{if $m$ is even}, \end{cases}
    \end{equation}
    where $\Gamma_{\mathbb{R}}(s) := \pi^{-\frac{s}{2}} \Gamma\left(\frac{s}{2}\right),$ $\Gamma_{\mathbb{C}}(s) := \Gamma_\R(s)\Gamma_\R(s+1) = 2(2\pi)^{-s}\Gamma(s)$, and $r \in \{0, 1\}$ is chosen so that $r \equiv \frac{m}{2} \pmod{2}$.  
    \item[(ii)] The completed $L$-function \begin{align*}
        \Lambda(s, \Sym^m f) := q_{\Sym^m f}^{\frac{s}{2}} \gamma(s, \Sym^m f) L(s, \Sym^m f)
    \end{align*}
    is entire of order 1. 
    \item[(iii)] There exists an $\epsilon_{\Sym^m f} \in \{1, -1\}$ such that $\Lambda(s, \Sym^m f)$ satisfies the functional equation
        \begin{equation} \label{eq:functional_equation}
            \Lambda(s, \Sym^m f) = \epsilon_{\Sym^m f} \Lambda(1-s, \Sym^m f).
        \end{equation}
\end{itemize}

\noindent
We note that $L(s,\Sym^m f)$ has infinitely many \textit{trivial} zeros on the nonnegative real axis corresponding precisely to the poles of $\gamma(s,\Sym^m f)$, and infinitely many \textit{nontrivial} zeros $\rho$ on the critical strip $0< \Re(\rho) < 1$; $L(s,\Sym^m f)$ is nonzero elsewhere.

Let $q_{\Sym^m f}$ denote the arithmetic conductor of $L(s, \Sym^m f)$. When $N$ is squarefree, we have $q_{\Sym^m f} = N^m$ \cite[Section 3]{CM}. When $f$ corresponds to a non-CM elliptic curve of arbitrary level, we have $q_{\Sym^m f} \leq N^{m+1}$. This is proven in Appendix \ref{ECappendix}. 

\begin{remark} \label{squarefree-remark}
In Theorem \ref{main2}, we require that $N$ is squarefree because no explicit upper bound of the form $q_{\Sym^m f} = N^{O(m)}$ is known in general. For a fixed newform $f$, it may be shown that $q_{\Sym^m f} = N^{O_f(m)}$, but this bound is not explicit \cite[Section 5]{R}. The weaker inequality $q_{\Sym^m f} \leq N^{\frac{3}{2}m^3}$ can be extracted from \cite[Lemma 2.1]{R}; however, this would alter the shape of the bound in our results.
\end{remark}

For $n \geq 1$, we define the function $\Lambda_{\Sym^m f}(n)$ by \begin{equation*}
    \sum_{n=1}^\infty \frac{\Lambda_{\Sym^m f} (n)}{n^s} := -\frac{L'}{L}(s, \Sym^m f), \qquad \Re(s) > 1.
\end{equation*}
From \eqref{eq:sym_power_L_Euler_prod}, we easily see that the values $\Lambda_{\Sym^m f}(n)$ are given explicitly by
\begin{equation}\label{eq:lambda_coeffs_def}
    \Lambda_{\Sym^m f}({n}) = \begin{cases} U_m(\cos (\ell \theta_p)) \log p & \text{if ${n} = p^{{\ell}}$ for some $p \nmid N$ and ${\ell} \geq 1$,} \\
    \sum_{j=0}^m \alpha_{j, \Sym^m f}(p)^{\ell} \, \log p & \text{if ${n} = p^{{\ell}}$ for some $p \mid N$ and $\ell \geq 1$,} \\
    0 & \text{otherwise}. \end{cases}
\end{equation}
In particular, this demonstrates that $    \lvert\Lambda_{\Sym^m f}(n)\rvert \leq (m+1) \Lambda(n)$, so that for $s = \sigma + it$ with $\sigma > 1$ and for all $m \geq 1$, we have
\begin{equation}\label{RTlemma}
             \left| \frac{L'}{L}(s, \Sym^m f) \right| \leq -(m+1) \frac{\zeta'}{\zeta}(\sigma).
\end{equation}
Additionally, it is well-known that $-\frac{L'}{L}(s, \Sym^m f)$ has real Dirichlet coefficients for any $m \geq 0$. In particular, this implies that the zeros of $L(s,\Sym^m f)$ come in complex conjugate pairs.

Since $\Lambda(s, \Sym^m f)$ is entire of order $1$ and is nonzero at $s = 0$, it admits a Hadamard product \begin{align*}
    \Lambda(s, \Sym^m f) = e^{A_{\Sym^m f} + B_{\Sym^m f}s}\prod_\rho \left(1-\frac{s}{\rho}\right)e^{-\frac{s}{\rho}}.
\end{align*} Combining this with the definition of $\Lambda(s, \Sym^m f)$, we obtain the formula \begin{align}
    \frac{\Lambda'}{\Lambda}(s, \Sym^m f) &= \frac{L'}{L}(s, \Sym^m f) + \frac{\gamma'}{\gamma}(s, \Sym^m f) + \frac{1}{2}\log q_{\Sym^m f} \nonumber \\
    &= B_{\Sym^m f} - \sum_\rho \left( \frac{1}{s-\rho} - \frac{1}{\rho} \right). \label{eq:log_deriv_hadamard_prod}
\end{align} It will be useful to note that by \cite[Proposition 5.7]{IK}, we have 
\begin{equation} \label{eq:B}
    \Re(B_{\Sym^m f}) = -\sum_{\rho} \Re\left( \frac{1}{\rho} \right).
\end{equation}

\section{Proof of Main Theorem} \label{section-main-theorem}

We first define the following analogue of the first Chebyshev function:
\[\vartheta_{f, I} (x) := \sum_{\substack{p \leq x\\ \theta_p \in I\\ p \nmid N}} \log p.\] 
We will deduce Theorems \ref{main2} and \ref{mainEC} from the following theorem.

 \begin{theorem}\label{main} 
Let $f$ be a non-CM holomorphic cuspidal newform of even integral weight $k\geq 2$, level $N,$ and trivial nebentypus such that $q_{\Sym^m f} \leq Q^{m+1}$ for all $m \geq 1$ for some $Q\geq 1$. Then for $\Cr{c_{99}}:= 58.084$, we have
\begin{align*}\left| \vartheta_{f,I}(x) - \mu_{\mathrm{ST}}(I) x \right| \leq \Cr{c_{99}} x \frac{\log((k-1)Q \log x)}{\sqrt{\log x}}\end{align*}
for any $x \geq 3$.
\end{theorem}

\begin{remark}
    If one is willing to sacrifice the effective range of this theorem, one can obtain constants smaller than $\Cr{c_{99}}$. More precisely, for any constant $\Cr{c_{99}}' > \frac{4\sqrt{2}}{2-\sqrt{3}} \approx 21.112$, our methods show that there exists an effectively computable constant $d > 0$, depending only on $\Cr{c_{99}}'$, such that Theorem \ref{main} applies for $x$, $k$, and $Q$ satisfying $\sqrt{\log{x}} \geq d\log((k-1)Q\log{x})$, replacing $\Cr{c_{99}}$ with $\Cr{c_{99}}'$.
\end{remark}

\noindent
We establish Theorem \ref{main} separately for small $x$ and for large $x$. The former case is treated by the following lemma. 

\begin{lemma} \label{triviality-bound}
Assume the hypotheses of Theorem \ref{main}. Then for $x \geq 0$, we have \begin{align*}
    \left|\vartheta_{f,I}(x) - \mu_{\ST}(I)x\right| \leq {\textstyle\left(1 + \frac{1}{36260}\right)}x.
\end{align*}
\end{lemma}
 \begin{proof}
In the case that $\mu_{\mathrm{ST}}(I)x > \vartheta_{f,I}(x)$, the statement  holds trivially. Otherwise, applying the bound $\vartheta(x) - x \leq \frac{1}{36260} x$ for $x>0$ from \cite{Du} and the observation that $\vartheta_{f,I}(x) \leq \vartheta(x)$ for all $x$, we arrive at the desired result.
 \end{proof}
This lemma implies Theorem \ref{main} for $x \leq 10^{100}$, as our claimed bound exceeds $(1+\frac{1}{36260})x$ for $x$ in this range. To bound $\vartheta_{f,I}(x)$ for large $x$, we will use the following approximation for the indicator function $\chi_I(\theta)$ of the interval $I$.

\begin{lemma}[\cite{RT}, Lemma 3.1]\label{lemma:trig_polys}
    For $I = [a, b] \subseteq [0, \pi]$ and $M$ a positive integer, there exist trigonometric polynomials
        \[
            F_{I,M}^\pm (\theta) = \sum_{m=0}^M \widehat{F}_{I, M}^\pm(m) U_m (\cos \theta)
        \]
    that satisfy the following properties:
        \begin{enumerate}
            \item For all $0 \leq \theta \leq \pi,$ we have $F_{I,M}^- (\theta) \leq \chi_I(\theta) \leq F_{I,M}^+(\theta)$.
            \item The constant term $\widehat{F}_{I,M}^{\pm}(0)$ satisfies $|\widehat{F}_{I,M}^{\pm}(0) - \mu_{\ST}(I)| \ \leq \frac{4}{M+1}.$
            \item For all $1 \leq m \leq M,$ the values $\widehat{F}_{I,M}^\pm(m)$ satisfy $|\widehat{F}_{I,M}^\pm(m)|\ \leq 4\left( \frac{1}{M+1} + \min\left\{ \frac{b-a}{2\pi}, \frac{1}{\pi m} \right\} \right)$.
        \end{enumerate}
\end{lemma}

We define the functions $\Theta_m(x)$ as 
\begin{equation*}
    \Theta_m(x) := \sum_{\substack{p \leq x \\ p \nmid N    }} U_m(\cos \theta_p) \log p.
\end{equation*}
To approximate $\vartheta_{f,I}(x),$ we will use the following variant of the Erd\H{o}s--Tur\'an inequality from \cite{RT}, which follows from Lemma \ref{lemma:trig_polys}:

\begin{lemma}\label{erdos-turan} 
If $M \geq 1$ is an integer and $1 < x-y < x$, then we have
\begin{align*}
    \bigg| \vartheta_{f,I}(x) - \mu_{\mathrm{ST}}(I) x \bigg| & \leq  \frac{4}{M+1} x + \left(\frac{1}{2} + \frac{2}{M+1}\right)y \nonumber \\
    &\quad + \max_{\pm} \Bigg\{\left(1 + \frac{4}{M+1}\right) \left| \pm \frac{1}{y}\int_{x}^{x\pm y} (\Theta_0(u)-u)du \right|\nonumber \\
    & \ \ \ \ \ + 4 \sum_{1 \leq m \leq M} \left(\frac{1}{M+1} + \frac{1}{\pi m}\right) \left| \pm \frac{1}{y} \int_x^{x\pm y} \Theta_m(u) du \right| \Bigg\}.
\end{align*}
\end{lemma}

\begin{proof}
To prove an upper bound on $\vartheta_{f,I}(x)-\mu_{\ST}(I)x$, we observe by monotonicity of $\vartheta_{f,I}(x)$ that 
\[
    \vartheta_{f,I}(x) - \mu_{\ST}(I)x \leq \frac{1}{y} \int_x^{x+y} \vartheta_{f,I}(u) \, du - \mu_{\ST}(I) x.
\]
Applying Lemma \ref{lemma:trig_polys} gives
\[
    \frac{1}{y}\int_{x}^{x+y} \vartheta_{f,I}(u)\, du \leq \sum_{m=0}^M \left( \widehat{F}_{I,M}^+(m) \cdot \frac{1}{y} \int_x^{x+y} \Theta_m(u)\, du \right).
\]
We obtain the stated result by using Lemma \ref{lemma:trig_polys} to bound each $\widehat{F}_{I,M}^+(m)$. The lower bound for $\vartheta_{f,I}(x)-\mu_{\ST}(I)x$ follows similarly.
\end{proof}

To estimate the terms in Lemma \ref{erdos-turan}, we use the following lemmas. 

\begin{lemma} \label{lemma:theta0} 
For all $x, y \geq 0$ such that $x-y \geq 3$, we find that
\begin{align*} \left| \frac{1}{y} \int_{x}^{x \pm y} (\Theta_0(u)-u)\ du\right| \leq \log N + \Cr{c_{60}}\frac{x+y}{\log(x+y)},
\end{align*} where $\Cr{c_{60}} := 1.2323$.
\end{lemma}

\begin{proof}
We have
\begin{align*}
\left| \frac{1}{y} \int_{x}^{x \pm y} (\Theta_0(u)-u)\ du \right| & \leq \bigg|\frac{1}{y} \int_{x}^{x \pm y} \bigg| \vartheta(u)-u-\sum_{\substack{p\leq u\\ p\mid N}} \log p \bigg|\ du \bigg| \\
&\leq 1.2323\frac{x+y}{\log(x+y)} + \log{N},
\end{align*} where in the last line we use the bound $|\vartheta(u)-u| \leq 1.2323  \cdot \frac{u}{\log u}$ for $u \geq 2$ due to Dusart \cite[Theorem 5.2]{Du}.
\end{proof}

\begin{lemma} \label{lemma4.3}
For $x, y > 0$ and $m \geq 1$, we have
\begin{align} 
\left| \frac{1}{y} \int_x^{x + y} \Theta_m(u) du \right| & \leq \left| \frac{1}{y} \int_{x}^{x + y} \sum_{n\leq u} \Lambda_{\Sym^m f}(n)\ du\right|  + (m+1)\log N \nonumber \\
& \quad +\left(1 + \frac{1}{36260}\right) (m+1)\sqrt{x}\left(1+\frac{y}{x}\right)\log(x+y).
\end{align}
\end{lemma}

\begin{proof}
Comparing the coefficients $\Lambda_{\Sym^m f}(u)$ as given by \eqref{eq:lambda_coeffs_def} with the definition of $\Theta_m(u)$ at ramified and unramified primes, we find that
    \begin{align} \label{thetalambda2}
        & \left|\int_{x}^{x+y} \Theta_m(u) \  du - \int_{x}^{x + y} \sum_{n \leq u } \Lambda_{\Sym^m f}(n) \, du \right| \nonumber\\
        & \ \ \ \leq \int_x^{x+y} \Bigg( \Bigg| \sum_{\substack{p|N \\ p \leq u}} \Lambda_{\mathrm{Sym}^m f}(p)\Bigg| + \Bigg| \sum_{\ell=2}^\infty \sum_{\substack{p^\ell \leq u}} \Lambda_{\mathrm{Sym}^m f}(p^\ell)\Bigg| \Bigg)\,du.
    \end{align} 
The first term in the integrand of \eqref{thetalambda2} satisfies
\begin{align*}
    \Bigg|\sum_{\substack{p|N \\ p \leq u}} \Lambda_{\mathrm{Sym}^m f}(p)\Bigg| \leq (m+1)\sum_{p \mid N} \log{p} \leq (m+1)\log{N},
\end{align*} 
and the second term satisfies
\begin{align*}
    \Bigg| \sum_{\ell=2}^\infty \sum_{\substack{p^\ell \leq u}} \Lambda_{\mathrm{Sym}^m f}(p^\ell)\Bigg| &\leq (m+1) \sum_{\ell=2}^{\lfloor\log{u}\rfloor} \vartheta(u^{\frac{1}{\ell}}) \leq \left(1+{\textstyle \frac{1}{36260}}\right)(m+1)u^{\frac{1}{2}}\log(u),
\end{align*} 
where we again use the bound $\vartheta(x) - x < \frac{1}{36260} \, x$ for $x > 0$ due to Dusart \cite[Section 1]{Du}. These bounds combine to give the desired result.
\end{proof}

Our estimates of the main term of the integral in Lemma \ref{lemma4.3} rely upon the following observation, which follows from the definition of $\Lambda_{\Sym^m{f}}(n)$ by applying standard contour integration techniques to evaluate the integral $-\frac{1}{2\pi i}\int_{2-i\infty}^{2+i\infty}\frac{L'}{L}(s, \Sym^m{f})\frac{(x+y)^s - x^s}{s(s+1)}\,ds$.
\begin{lemma} \label{integral}
For $m\geq 1$, we have 
   \begin{align*}
    \int_x^{x+y} \sum_{n \leq u} \Lambda_{\mathrm{Sym}^m f}(n)\,du &= R_1(x, y, \Sym^m f) + R_2(x, y, \Sym^m f) + R_3(x, y, \Sym^m f),
\end{align*} where \begin{align*}
    R_1(x, y, \Sym^m f) := -\!\!\!\!\!\!\! &\sum_{\rho\ \text{nontrivial}}\!\!\!\!\!\!\! \frac{(x+y)^{\rho+1} - x^{\rho+1}}{y\rho(\rho+1)}, \quad R_2(x, y, \Sym^m f) := -\!\!\!\sum_{\substack{\rho\ \text{trivial} \\ \rho \neq 0, -1}}\!\!\! \frac{(x+y)^{\rho+1} - x^{\rho+1}}{y\rho(\rho+1)}, \\
    R_3(x, y, \Sym^m f) &:= -\frac{1}{y}(\mathrm{Res}_{s=0} + \mathrm{Res}_{s=-1}) \left(\frac{L'}{L}(s, \mathrm{Sym}^m f)\frac{(x+y)^{s+1}-x^{s+1}}{s(s+1)}\right).
\end{align*}
Here, the sums are taken over the nontrivial and trivial zeros (as defined in Section \ref{section-symmetric-powers}) of $L(s,\mathrm{Sym}^m f)$.
\end{lemma}

\begin{proof}
We first note that \begin{align*}
    -\frac{1}{2\pi i}\int_{2-i\infty}^{2+i\infty} \frac{L'}{L}(s, \Sym^m{f}) \frac{(x+y)^s - x^s}{s(s+1)}\,ds &= \sum_{n=1}^\infty \frac{\Lambda_{\Sym^m{f}}(n)}{2\pi i}\int_{2-i\infty}^{2+i\infty} \frac{\left(\frac{x+y}{n}\right)^s-\left(\frac{x}{n}\right)^s}{s(s+1)}\,ds \\
    &= \int_x^{x+y} \sum_{n \leq u} \Lambda_{\mathrm{Sym}^m f}(n)\,du,
\end{align*} where the first equality follows because $\sum_{n=1}^\infty |\Lambda_{\Sym^m{f}}(n)| \int_{2-i\infty}^{2+i\infty} \frac{|(x+y)^s|+|x^s|}{|n^s s(s+1)|}\,ds$ converges. Meanwhile, for any $H \in \frac{1}{4}\mathbb{N} \setminus \frac{1}{2}\mathbb{N}$ and $K > 0$ not equal to the imaginary part of any nontrivial zero of $L(s, \Sym^m{f})$, we may write \begin{align*}
    &-\frac{1}{2\pi i}\left(\int_{2-iK}^{2+iK} + \int_{2+iK}^{-H+iK} + \int_{-H+iK}^{-H-iK} + \int_{-H-iK}^{2-iK}\right) \frac{L'}{L}(s, \Sym^m{f})\frac{x^s}{s(s+1)}\,ds \\
    &\ \ \ \ \ = \sum_{\substack{-H < \Re(\rho) < 2 \\ -K < \Im(\rho) < K}} \Res_{s=\rho}\left(\frac{L'}{L}(s, \Sym^m{f})\frac{(x+y)^{s+1}-x^{s+1}}{s(s+1)}\right),
\end{align*} where the sum extends over all poles of $\frac{L'}{L}(s, \Sym^m{f})\frac{(x+y)^{s+1}-x^{s+1}}{s(s+1)}$ satisfying the specified conditions. From here, we may use \eqref{eq:log_deriv_hadamard_prod} to bound the contribution from the upper, left, and lower legs of integration as we take $H \to \infty$ and then $K \to \infty$ by considering the contribution from each term. These contributions go to zero, so we obtain the stated result.
\end{proof}

In Section \ref{Section:PerronIntegral}, we will prove the following bounds for $|R_1(x, y, \Sym^m f)|$, $|R_2(x, y, \Sym^m f)|$, and $|R_3(x, y, \Sym^m f)|$, which apply for $f$ satisfying the hypotheses of Theorem \ref{main}. Here, the $c_i$ are absolute constants whose values are stated in Appendix \ref{appendixconstants}.
    \begin{lemma}\label{asymptoticnontrivial} 
    For $T \geq 200$ and $m \geq 1,$ the contribution $R_1(x, y, \Sym^m f)$ from the nontrivial zeros satisfies
    \begin{align*}
        |R_1(x,y, \Sym^m f)| &\leq (m+1)\log((k-1)Q(m+2)T) \Bigg( x^{1-\eta_m(T)} \Cr{c_{233}}  \log((k-1)Q(m+2)T)\nonumber\\ 
        &\qquad \qquad \qquad +  {yx^{-\eta_m(T)}}\, \frac{\Cr{c_{303}}}{2}\, T  + x\left(\frac{2x}{y} + 2 + \frac{y}{x}\right) \frac{\Cr{c_{220}}}{T} \Bigg) \\
        & \qquad \qquad \qquad + x^{1-\eta_m(1)}\, \frac{\Cr{c_{310}} }{\Cr{zfr5}} (m+1)(m+7)^2 \log^2 (\Cr{zfr3} (k-1) Q (m+7)),
    \end{align*} 
    where $\eta_{m}(T) := \frac{\Cr{zfr5}}{(m+7)^2\log\left(\Cr{zfr3}(k-1)Q(m+7)\sqrt{T^2+1}\right)}$.
    \end{lemma}

\begin{lemma}\label{trivialbound} \label{trivial} 
The contribution $R_2(x, y, \Sym^m f)$ from the trivial zeros $\rho\neq 0,-1$ satisfies
\begin{equation*}
    |R_2(x,y, \Sym^m f)| = \left|\sum_{\substack{\rho\ \text{trivial} \\ \rho \neq 0, -1}} \frac{(x+y)^{\rho+1} - x^{\rho+1}}{y\rho(\rho+1)}\right|\leq  \frac{2(m+1)}{3{\sqrt{x}}}.
\end{equation*}
\end{lemma}

\begin{lemma}\label{residues} 
The contribution $R_3(x, y, \Sym^m f)$ from the residues at $s=0,-1$ satisfies
\begin{align*}
    |R_3(x,y, \Sym^m f)| & \leq \frac{\Cr{c_{310}} }{\Cr{zfr5}} (m+1) \log((k-1)Q(m+2)){(m+7)^2 \log(\Cr{zfr3}(k-1)Q(m+7))} \nonumber \\
    &\quad + 3(m+1) \log((k-1)Q(m+1))  + \Cr{reslast}(m+1) +  \log(x+y) + \frac{1}{x}.
\end{align*}
\end{lemma} 

Finally, we are equipped to prove the main theorem of this section.

\begin{proof}[Proof of Theorem \ref{main}]
We bound the various terms in Lemma \ref{erdos-turan} using Lemmas \ref{lemma:theta0}, \ref{lemma4.3}, \ref{integral}, \ref{asymptoticnontrivial}, \ref{trivialbound}, and \ref{residues}. It then suffices to choose $M$, $T,$ and $y$ appropriately in terms of $x$ to balance the contributions of the low- and high-lying zeros.

We first define the function \begin{equation}
    \mathcal{M}_f(x):=\frac{\sqrt{\Cr{zfr5}\log{x}}}{2\log((k-1)Q \log{x})},
\end{equation} and for a given $x, k, Q$, fix the quantities 
\begin{equation*}
    M = \left\lfloor  \mathcal{M}_f(x) - 7\right\rfloor,\quad \quad T = \frac{2 (k-1)Q}{\Cr{zfr3}\sqrt{1.01\Cr{zfr5}}} (\log{x})^{\frac{3}{2}} \log((k-1)Q \log{x}).
\end{equation*}

Substituting these quantities into Lemma \ref{erdos-turan} requires $M \geq 1$, which may only hold for large $x$ ($T\geq 1$, while also necessary, holds in all relevant cases; in fact, we freely use $ T \geq 200$, which holds for $x \geq 4$). In particular, we may assume that $x$ exceeds $10^{100}$; using Lemma \ref{triviality-bound}, the bound in Theorem \ref{main} holds trivially for smaller values of $x$. 

With these choices of $M$ and $T$, we can bound $x^{-\eta_m(T)}$ for $m \leq M$ by noting that for $T \geq 10,$ we have
\begin{align*}
    \exp\left(-(\log{x}) {\eta_{m}(T)}\right) & \leq  \exp \left(- \frac{\Cr{zfr5}\log{x}}{\mathcal{M}_f(x)^2\log(\Cr{zfr3}\sqrt{1.01}(k-1)Q \mathcal{M}_f(x) T )}\right) \leq ((k-1)Q \log{x})^{-2}.
\end{align*} We may also bound $x^{-\eta_m(1)}$ directly, via
\begin{align*}
    \exp\left( -(\log{x})\eta_m(1) \right) & \leq \exp\left(- \frac{\Cr{zfr5}\log{x}}{\mathcal{M}_f(x) ^2 \log(\sqrt{2}\Cr{zfr3}(k-1)Q\mathcal{M}_f(x) )}\right) \leq ((k-1)Q\log{x})^{-4}.
\end{align*} Finally, we choose the value
\begin{equation*}
     y =  \frac{x{\Cr{zfr3}\sqrt{1.01\Cr{zfr5}}}}{2\sqrt{\log{x}}\log((k-1)Q\log{x})}  ,
\end{equation*}
which is approximately ${T^{-1}x^{1+\eta_M(T)/2}}$, to balance the contributions from $yx^{-\eta_M(T)}$ and $\frac{x^2}{yT}$.

Using these values, we are prepared to complete our calculation by applying our various bounds on $|\int_x^{x+y}\Theta_m (u)du|$. Note that for $x-y \geq 2$, we may make the change of variables $x \mapsto x-y, y \mapsto y$ to also bound the integral $|\int_{x-y}^x \Theta_m(u)du|$ using the same methods. In particular, Lemmas \ref{lemma4.3} and \ref{integral} give us the following bound, for $m \geq 1$:
\begin{align} \label{eq:pm_theta_integral}
    \left| \frac{1}{y} \int_x^{x\pm y} \Theta_m(u) du \right| & \leq \max\bigg\{ \sum_i |R_i(x,y, \Sym^m f)| + {\textstyle \left(1+\frac{1}{36260}\right)} (m+1)\sqrt{x}\left(1+\frac{y}{x}\right)\log(x+y), \nonumber\\
    & \quad \qquad  \sum_i |R_i(x-y,y, \Sym^m f)|  +  {\textstyle \left(1+\frac{1}{36260}\right)} (m+1)\sqrt{x}\left(1+\frac{y}{x}\right)\log(x+y) \bigg\} \nonumber \\
    & \quad + (m+1)\log(N).
\end{align}

\noindent Using Lemma \ref{asymptoticnontrivial} and the fact that $(x-y)^{1-\eta_m(T)} \leq x^{1-\eta_m(T)}$, we observe that
\begin{align*}
    \max\bigg\{ & |R_1(x,y, \Sym^m f)|, |R_1(x-y, y, \Sym^m f)| \bigg\} \\
    &\leq (m+1)\log((k-1)Q(m+2)T) \Bigg( x^{1-\eta_m(T)} \Cr{c_{233}}  \log((k-1)Q(m+2)T)\nonumber\\ 
        &\qquad \qquad \qquad +  {\frac{y}{x-y}\,x^{1-\eta_m(T)}}\, \frac{\Cr{c_{303}}}{2}\, T  + \left(\frac{2x^2}{y} + 2x + {y}\right) \Cr{c_{220}}\frac{1}{T} \Bigg) \\
        & \qquad \qquad + x^{1-\eta_m(1)}\, \frac{\Cr{c_{310}}}{\Cr{zfr5}}\, (m+1)(m+7)^2 \log^2 (\Cr{zfr3} (k-1) Q (m+7)),
\end{align*}
We then bound this term for $m \leq M$ by replacing any $m+a$ inside a logarithm with $\mathcal{M}_{f,I}(x)$, which produces a bound on each term $\max\{|R_1(x, y, \Sym^m f)|, |R_1(x-y, y, \Sym^m f)|\}$ in the sum $\sum_{m=1}^M \left( \frac{1}{M+1} + \frac{1}{\pi m} \right) \Big| \frac{1}{y}\int_x^{x\pm y} \Theta_m(u) du\Big|$ in Lemma \ref{erdos-turan}.

We bound the other terms in \eqref{eq:pm_theta_integral} similarly. Lemma \ref{trivial} shows that 
    \begin{align*}
        |\max\{R_2(x,y, \Sym^m f), R_2(x-y, y, \Sym^m f)\}| \leq \frac{2}{3} \cdot \frac{(m+1)}{\sqrt{x-y}},
    \end{align*}
and Lemma \ref{residues} shows that
    \begin{align*}
        |\max & \{R_3(x,y, \Sym^m f), R_3(x-y, y, \Sym^m f)\}| \\
        &\leq \frac{\Cr{c_{310}} }{\Cr{zfr5}} (m+1) \log((k-1)Q(m+2)){(m+7)^2 \log(\Cr{zfr3}(k-1)Q(m+7))} \nonumber \\
        &\quad + 3(m+1) \log((k-1)Q(m+1))  + \Cr{reslast}(m+1) +  \log(x+y) + \frac{1}{x-y}.
    \end{align*}

\noindent As most of our summands will have a factor of $m+1$, we will use the simplification
\begin{equation*}
    \sum_{m=1}^M \left( \frac{1}{M+1}+ \frac{1}{\pi m}\right)(m+1) \leq  \left( 1 + \frac{2}{\pi}\right)\mathcal{M}_f(x),
\end{equation*}
as well as the bound $\sum_{m=1}^M \left( \frac{1}{M+1} + \frac{1}{\pi m} \right) \leq 1 + \frac{1}{\pi}\left(\log(\mathcal{M}_f(x)) + 1 \right)$ throughout.

With this in mind, we may bound $|\vartheta_{f,I}(x)-\mu_{\ST}(I)x|$ using Lemma \ref{erdos-turan}. For any $M_0 \geq 1$, if we choose $x$ large enough so that $M \geq M_0$, then a computer-assisted calculation allows us to find a constant $c(M_0)$ so that the 
all terms in our bound are absorbed into the leading term, 
which then takes the form \[c(M_0)\frac{\log((k-1)Q\log{x})}{\sqrt{\log x}},\] valid for \[\frac{\sqrt{\log{x}}}{\log((k-1)Q\log{x})} \geq \frac{2(M_0 + 7)}{\sqrt{\Cr{zfr5}}}.\] 
By Lemma \ref{triviality-bound}, Theorem \ref{main} will hold trivially for \[ 1 + \frac{1}{36260}\leq \frac{c(M_0)\log((k-1)Q\log{x})}{\sqrt{\log{x}}}.\] We obtain the desired value by setting these two quantities equal to each other and solving numerically for $M_0$ and $c(M_0)$, giving us the values $M_0=4$ and $c(M_0)\leq 58.084$, so that the result in Theorem \ref{main} holds for any $x \geq 3$.
\end{proof}
\begin{proof}[Proof of Theorem \ref{main2} and \ref{mainEC}] 
Write $\vartheta_{f,I}(x) =\mu_{\ST}(I)x + \varepsilon(x)$, where by the results of Theorem \ref{main} we may use the bound 
\begin{equation*}
    |\varepsilon(x)| \leq \Cr{c_{99}}x\frac{\log((k-1)Q\log{x})}{\sqrt{\log{x}}}. 
\end{equation*} By partial summation, we have 
\begin{align*}
    \pi_{f,I}(x) 
    & = \mu_{ST}(I) \left( \Li(x) + \frac{2}{\log 2} \right) + \frac{\varepsilon(x)}{\log x} + \int_2^x \frac{\varepsilon(u)}{u \log^2 u} du.
\end{align*}
For $x \geq 13$, we can bound the last term from the partial summation as
\begin{align} \label{eq:partial_summation_error2}
    \int_2^x \frac{\varepsilon(u)}{u \log^2 u} du &= \int_2^x \Cr{c_{99}}\, \frac{\log ((k-1)Q \log u)}{\sqrt{\log u}} \,\frac{1}{\log^2 u} \, du \nonumber \\
    & \leq \Cr{c_{99}} \frac{\log((k-1)Q \log 2)}{\sqrt{\log 2}} \left(\Li(x)-\frac{x}{\log x} + \frac{2}{\log 2}\right).
\end{align}

As stated, the result of Theorem \ref{main2} and \ref{mainEC} is trivially true for $x \leq 10^{100}$, as \[\frac{\log((k-1)Q \log{x})}{\sqrt{\log{x}}}> \frac{1}{58}\] in this range.
Assuming $x >10^{100}$, 
we may apply the bounds $\frac{x}{\log x} < \pi(x)$ for $x \geq 5393$ due to Dusart \cite{Du}, and  \[\left|\pi(x) - \Li(x)\right| \; \leq \;  0.2795 \frac{x}{(\log x)^{3/4}} \exp\left( -\sqrt{\frac{\log x}{6.455}} \right)\] for $x \geq 229$ due to Trudgian \cite{Tr2}. Using these results, we see that
\begin{align*}
    |\pi_{f, I}(x)-\mu_{\ST}(I) \pi(x)| & \leq |\pi_{f, I}(x) - \mu_{\ST}(I) \Li(x)| + |\Li(x) - \pi(x)| \\
    & \leq 1.000015\, \Cr{c_{99}} \pi(x)\frac{\log((k-1)Q\log{x})}{(\log{x})^{\frac{1}{2}}} ,
\end{align*}
where the last bound holds for $x \geq 10^{100}$.
\end{proof}

\section{An Explicit Zero-Free Region}\label{horizontalsection}

The primary purpose of this section is to prove an explicit zero-free region for the functions $L(s, \Sym^m f)$. 
As in \cite{T}, we do this by constructing several auxiliary products of $L$-functions with shifted arguments, such as 
\[
    \zeta(s)^3 \zeta(s + 2i \gamma)\zeta(s-2i\gamma)L(s + i\gamma,\Sym^m f)^2 L(s-i\gamma, \Sym^m f)^2 \prod_{j=1}^m L(s, \Sym^{2j} f).
\]
These auxiliary products will have all nonnegative Dirichlet coefficients.

Since we apply similar reasoning to each auxiliary product,
we find it convenient to encapsulate analytic information about $L$-functions, $L$-functions with shifted arguments, and products of $L$-functions with shifted arguments in a single structure. This will allow us to state the main result that we use to prove our explicit zero-free region in terms of the \textit{analytic conductor} of such functions, which we will define shortly.

With this in mind, let $m \geq 0$, and let $\mathcal{L}$ be the tuple consisting of the following data: \begin{itemize}
    \item[(i)] a positive integer $q_{\mathcal{L}}$;
    \item[(ii)] a complex number $\epsilon_{\mathcal{L}}$ with modulus $1$;
    \item[(iii)] a sequence $\kappa_0(\mathcal{L}), \dots, \kappa_m(\mathcal{L})$ of complex numbers with positive real part, up to reordering; and
    \item[(iv)] a sequence $\alpha_0(\cdot, \mathcal{L}), \dots, \alpha_n(\cdot, \mathcal{L})$ of functions from the primes to complex numbers with modulus at most $1$, up to reordering.
\end{itemize} For $\Re(s) > 1$, write \begin{align*}
    L(s, \mathcal{L})\! :=\! \prod_{p} \prod_{j=0}^m \frac{1}{1-\alpha_j(p, \mathcal{L})p^{-s}},\quad 
    \gamma(s, \mathcal{L})\! :=\! \prod_{j=0}^m \Gamma_\R\left(s+\kappa_j(\mathcal{L})\right), \quad 
    \Lambda(s, \mathcal{L}) := q_{\mathcal{L}}^{\frac{s}{2}}\gamma(s, \mathcal{L})L(s, \mathcal{L}).
\end{align*} Moreover, set $\overline{\mathcal{L}} := (q_{\mathcal{L}}, \overline{\epsilon}_{\mathcal{L}}, \{\overline{\kappa}_j(\mathcal{L})\}_0^m, \{\overline{\alpha}_j(\mathcal{L})\}_0^m)$, and suppose that $\Lambda(s, \mathcal{L})$ has a meromorphic analytic continuation satisfying \begin{align}\label{zfr functional eq}
    \Lambda(1-s, \mathcal{L}) &= \epsilon_{\mathcal{L}}\,\Lambda(s, \overline{\mathcal{L}})
\end{align} and admitting a product representation \begin{align*}
    \Lambda(s, \mathcal{L}) &= e^{A_{\mathcal{L}}+B_{\mathcal{L}}s}\left(\prod_\rho\left(1 - \frac{s}{\rho}\right)e^{\frac{s}{\rho}}\right)\left(\prod_\omega\frac{1}{s-\omega}\right),
\end{align*} where the zeros $\rho$ satisfy $0 < \Re{\rho} < 1$ and the poles $\omega$ satisfy $\Re{\omega} \in \{0, 1\}$. In such a case, we will call such an $\mathcal{L}$ an \textit{$L$-tuple of degree $m+1$}, and we will define $\Lambda_{\mathcal{L}}(n)$ by the property that \begin{align*}
    \sum_{n=1}^\infty \Lambda_{\mathcal{L}}(n)n^{-s} := -\frac{L'}{L}(s, \mathcal{L}).
\end{align*}

We start with the following lemma, which gives a bound for the real part of the logarithmic derivative of $\Gamma_\R(s)$.

\begin{lemma}
Let $s = \sigma+it$ with $\sigma \geq \frac{1}{2}$. Then we have \begin{align*}
    \Re\left(\frac{\Gamma_\R'}{\Gamma_\R}(s)\right) \leq -\frac{\Cr{zfr1}}{2} + \frac{1}{4}\log(\sigma^2+t^2),
\end{align*} where $\Cr{zfr1} := \log{\pi} + \gamma > 1.721$ and $\gamma$ denotes the Euler--Mascheroni constant.
\end{lemma}

\begin{proof}
Start by using the identity $\frac{\Gamma'}{\Gamma}(z) = -\gamma - \frac{1}{z} + \sum_{j=1}^\infty \frac{z}{j(j+z)},$ valid for $z \notin \Z_{\leq 0}$, to write \begin{align*}
    \Re\left(\frac{\Gamma_\R'}{\Gamma_\R}(s)\right) &= -\frac{\gamma}{2}-\frac{\log{\pi}}{2}-\frac{\sigma}{\sigma^2+t^2} + \sum_{j=1}^\infty \frac{\sigma(2j+\sigma)+t^2}{2j((2j+\sigma)^2+t^2)}.
\end{align*} Since the summand on the right is convex for $j \in \R_{\geq 0}$, we find that \begin{align*}
    \Re\left(\frac{\Gamma_\R'}{\Gamma_\R}(s)\right) &\leq -\frac{\gamma+\log{\pi}}{2}-\frac{\sigma}{\sigma^2+t^2} + \int_{\frac{1}{2}}^\infty \frac{\sigma(2x+\sigma)+t^2}{2x((2x+\sigma)^2+t^2)}\,dx \\
    &\leq -\frac{\gamma+\log{\pi}}{2} + \frac{1}{4}\log(\sigma^2+t^2). \qedhere
\end{align*}
\end{proof}

Define the \textit{analytic conductor} of an $L$-tuple $\mathcal{L}$ of degree $m+1$ by \begin{align*}
    C(\mathcal{L}) &:= q_{\mathcal{L}}\prod_{j=0}^m|1+\kappa_j(\mathcal{L})|.
\end{align*} 
The following theorem gives a zero-free region for certain types of $L$-functions in terms of the analytic conductor of an auxiliary $L$-tuple.

\begin{theorem} \label{general-zfr-theorem}
Let $\mathcal{L}$ be an $L$-tuple of degree $m+1$ such that $\overline{\mathcal{L}} = \mathcal{L}$ (allowing for a potential re-ordering of the coefficients $\alpha_j, \kappa_j$) and $\Lambda_{\mathcal{L}}(n)$ is non-negative for all $n \geq 1$, and let $t > 0$ and $0 \leq \delta < \frac{1}{2}$. Moreover, suppose that \begin{itemize}
    \item[(i)] the function $L(s, \mathcal{L})$ is entire except for possible poles of order at most $a_0$ at $s = 1$ and at most $a_1$ at both $s = 1 \pm 2it$, and
    \item[(ii)] the function $L(s, \mathcal{L})$ has a zero of order at least $b_0$ at $s = 1-\delta$ and zeros of order at least $b_1$ at both $s = 1-\delta \pm 2it$,
\end{itemize} where $b_0 \geq a_0 > 0$, $b_1 \geq a_1 \geq 0$, and $\frac{b_1}{b_0} \geq \frac{a_1}{a_0}$. Then we have 
\begin{align*}
    \delta \geq \frac{2(\sqrt{b_0} - \sqrt{a_0})^2}{\log{C(\mathcal{L})} + 2\Cr{zfr1}^{-1}\sqrt{a_0}(\sqrt{b_0}-\sqrt{a_0})}.
\end{align*}
\end{theorem}

\begin{proof}
Let $\varepsilon > 0$. Taking the logarithmic derivatives of the two formulas we have for $\Lambda(s, \mathcal{L})$ at $s=1+\varepsilon$ yields 
\begin{align*}
    &\sum_\rho \Re\left(\frac{1}{1+\varepsilon-\rho}\right) - \sum_\omega \Re\left(\frac{1}{1+\varepsilon-\omega}\right) \\
    &\quad = \frac{1}{2}\log(q_{\mathcal{L}}) + \sum_{j=0}^m \frac{\Gamma_\R'}{\Gamma_\R}(\varepsilon+1+\kappa_j) - \sum_p \sum_{\ell \geq 1} \Lambda_{\mathcal{L}}(p)p^{-\ell s} \\
    & \quad \leq \frac{1}{2}\left(-(m+1)\Cr{zfr1} + \log{C(\mathcal{L})} + \frac{1}{2}\sum_{j=0}^m\log\left(1 + \frac{2\Re(1+\kappa_j)\varepsilon+\varepsilon^2}{|1+\kappa_j|^2}\right)\right) \\
    & \quad \leq \frac{1}{2}(-(m+1)\Cr{zfr1} + \log{C(\mathcal{L})} + (m+1)\varepsilon),
\end{align*} where we use that $\Re(B) =- \sum_\rho \frac{1}{\rho}$ as $\mathcal{L}$ satisfies the functional equation \eqref{zfr functional eq}, following the argument from \cite[Proposition 5.7]{IK}. In the context of this theorem, we obtain \begin{align*}
    &\frac{2b_0}{\varepsilon+\delta} + \frac{4b_1(\varepsilon+\delta)}{(\varepsilon+\delta)^2+4t^2} + \frac{2b_0}{\varepsilon+1-\delta} + \frac{4b_1(\varepsilon+1-\delta)}{(\varepsilon+1-\delta)^2+4t^2} \\
    &\quad \leq \frac{2a_0}{\varepsilon} + \frac{4a_1\varepsilon}{\varepsilon^2+4t^2} + \frac{2a_0}{\varepsilon+1} + \frac{4a_1(\varepsilon+1)}{(\varepsilon+1)^2+4t^2} + \log{C(\mathcal{L})} + (m+1)(\varepsilon - \Cr{zfr1}),
\end{align*} 
from which we see that \begin{align} \label{eq:main_zfr_ineq}
    0 &\leq \frac{2\varepsilon+1}{\varepsilon+1}\left(\frac{2a_0}{\varepsilon}-\frac{2b_0}{\varepsilon+\delta}\right) + \frac{(2\varepsilon+1)((\varepsilon+\frac{1}{2})^2+4t^2-\frac{1}{4})}{(\varepsilon+1)^2+4t^2}\left(\frac{4a_1}{\varepsilon^2+4t^2}-\frac{4b_1}{(\varepsilon+\delta)^2+4t^2}\right) \nonumber \\
    &\quad \quad + \log(C(\mathcal{L})) + (m+1) (\varepsilon-\Cr{zfr1}).
\end{align} 
Now set $d := \log(C(\mathcal{L}))+2\Cr{zfr1}^{-1}\sqrt{a_0}(\sqrt{b_0}-\sqrt{a_0})$, choose $\varepsilon = \frac{2}{d}\sqrt{a_0}(\sqrt{b_0}-\sqrt{a_0})$, and suppose that $\delta < \frac{2}{d}(\sqrt{b_0}-\sqrt{a_0})^2$. Then we have 
\begin{align}\label{eq:main_zfr_ineq_term1}
    \frac{2a_0}{\varepsilon}-\frac{2b_0}{\varepsilon+\delta} < -d
\end{align} and 
\begin{align} \label{eq:main_zfr_ineq_term2}
    \frac{4a_1}{\varepsilon^2+4t^2}-\frac{4b_1}{(\varepsilon+\delta)^2+4t^2} &< \frac{b_1}{b_0}\left(\frac{1}{\frac{(\sqrt{b_0}-\sqrt{a_0})^2}{d^2}+\frac{t^2}{a_0}} - \frac{1}{\frac{(\sqrt{b_0}-\sqrt{a_0})^2}{d^2}+\frac{t^2}{b_0}}\right)\leq 0,
\end{align} using the fact that $\frac{b_1}{b_0} \geq \frac{a_1}{a_0}$. Hence, using \eqref{eq:main_zfr_ineq_term1} and \eqref{eq:main_zfr_ineq_term2} in \eqref{eq:main_zfr_ineq}, we find that
\begin{align*}
    0 < -d + \log{C(\mathcal{L})} + (m+1)\left(\varepsilon-\Cr{zfr1}\right) &\leq -2\Cr{zfr1}^{-1}\sqrt{a_0}(\sqrt{b_0}-\sqrt{a_0}),
\end{align*} a contradiction.
\end{proof}

To apply the above theorem to symmetric power $L$-functions, we start by defining $\mathcal{L}_{m}(f)$ as the $L$-tuple corresponding to the $m$-th symmetric power $L$-function. In addition, for $t \in \R$, we define $\mathcal{L}_m(f,t)$ as the $L$-tuple corresponding to $L(s+it, \Sym^m f)$. We start by bounding the analytic conductor of $\mathcal{L}_m(f,t)$.

\begin{lemma} \label{lemma:log_conductor_sym_power_L}
For $m \geq -1$, setting $\Cr{zfr2} := 6\log{2}-1 > 3.158$, we have \begin{align*}
    \log{C(\mathcal{L}_m(f,t))} &\leq \log q_{\Sym^m f} + (m+1)\log\left(\frac{1}{2e}(k-1)\sqrt{1+t^2}\right)  + (m+4)\log(m+4) - \Cr{zfr2}.
\end{align*} 
\end{lemma}

\begin{proof}
Setting $\mathcal{L} := \mathcal{L}_m(f)$ for convenience, we see that \begin{align*}
    \log{C(\mathcal{L}_m(f,t))} - \log{q_{\mathcal{L}}} 
    & \leq \frac{1}{2}\sum_{j=0}^m \log\left((1+\kappa_j(\mathcal{L}))^2(1+t^2)\right)\\ &= \frac{1}{2}(m+1)\log(1+t^2) + \sum_{j=0}^m \log(1+\kappa_j(\mathcal{L})).
\end{align*}  Define $\delta_2(m)$ to be $1$ for $m$ odd and 0 otherwise. Then we can write
\begin{align*}
    \sum_{j=0}^m \log(1+\kappa_j(\mathcal{L})) &= (1-\delta_2(m)) \log (r + 1) + \!\!\!\! \sum_{j=1}^{\frac{1}{2}(m+\delta_2(m))} \!\!\!\! \log\left( \left(j-{\textstyle\frac{\delta_2(m)}{2}}\right)(k-1)+1\right) \\
    & \ \ \ \   + \!\!\!\! \sum_{j=1}^{\frac{1}{2}(m+\delta_2(m))} \log\left(\left(j-{\textstyle \frac{\delta_2(m)}{2}}\right)(k-1)+2\right) \\ 
    & \leq (1 - \delta_2(m)) \log 2 + (m+\delta_2(m))\log(k-1) + 2\int_2^{\frac{m+\delta_2(m)}{2}+2}\!\!\!\!\!\!\!\!\log\left(x -{\textstyle \frac{\delta_2(m)}{2}}\right)\,dx \\
    & \leq (m+1) \log \left( {\textstyle \frac{k-1}{2e}} \right) + (m+4) \log (m+4) - (6\log{2} - 1). \qedhere
\end{align*}
\end{proof}

For $t > 0$, write $\mathcal{L}_{m \times m}$ for the $L$-tuple corresponding to \begin{align*}
    \prod_{j=1}^m L(s, \Sym^{2j} f) = \zeta(s)^{-1} L(s, \Sym^m f \times \Sym^m f),
\end{align*} where $L(s, \Sym^m f\, \times\, \Sym^m f)$ is the Rankin--Selberg $L$-function associated to $\Sym^m \pi_f \,\otimes \,\Sym^m \pi_f$. The following lemma bounds the analytic conductor of $\mathcal{L}_{m \times m}$.

\begin{lemma} \label{lemma:conductor_L_mxm}
Let $m \geq 0$, and suppose that there exists $Q \geq 1$ such that for all $1 \leq n \leq 2m$, we have $q_{\Sym^n f} \leq Q^{n+1}$. Then we have \begin{align*}
    \log{C(\mathcal{L}_{m \times m})} &\leq (m^2+2m)\log\left((k-1)Q\right) + (m+3)^2\log\left(\Cr{zfr3}(m+3)\right) - \Cr{zfr4},
\end{align*} where $\Cr{zfr3} := \frac{2}{\sqrt{e}} < 1.214$ and $\Cr{zfr4} := 9\log(6) - \frac{9}{2} > 11.62$.
\end{lemma}

\begin{proof}
Since the analytic conductor is multiplicative, we can use Lemma \ref{lemma:log_conductor_sym_power_L} to write \begin{align*}
    \log{C(\mathcal{L}_{m \times m})} = \sum_{j=1}^m \log{C(\mathcal{L}_{2j}(f))} 
    & \leq (m^2+2m)\log{Q} + (m^2+2m)\log\left(\frac{1}{2e}(k-1)\right) \\
    &\quad + \int_1^{m+1} (2x+4)\log(2x+4)dx - m\Cr{zfr2},
\end{align*}
from which the stated result follows.
\end{proof}

\begin{theorem} \label{theorem:zfr}
Let $\beta+i\gamma$ be a zero of $L_m(s, \mathcal{L}_m(f)) = L(s, \Sym^m f)$ with $\beta > \frac{1}{2}$ and $m \geq 1$, and suppose that $q_{\Sym^n(f)} \leq Q^{n+1}$ for all $1 \leq n \leq 2m+1$. Then we have 
\begin{align*}
    1 - \beta \geq \frac{\Cr{zfr5}}{(m+7)^2\log\left(\Cr{zfr3}(k-1)Q(m+7)\sqrt{1+\gamma^2}\right)},
\end{align*} where $\Cr{zfr5} := 2(2-\sqrt{3})^2 > 0.1435$.

\end{theorem}

\begin{proof}
To start, suppose $\gamma \neq 0$, and consider the $L$-tuple $\mathcal{A}_1$ with corresponding function 
\begin{align*}
    L(s, \Pi_1\!\!\times\!\widetilde{\Pi}_1) = \zeta(s)^3\zeta(s+2i\gamma)\zeta(s-2i\gamma)L(s+i\gamma, \Sym^m f)^2 L(s-i\gamma, \Sym^m f)^2\prod_{j=1}^m L(s, \Sym^{2j} f),
\end{align*}
where $\Pi_1 = |\det |^{i\gamma}\ \boxplus \ |\det |^{-i\gamma}\  \boxplus\ \Sym^m{f}$ and $\boxplus$ indicates the isobaric sum. Here, we used the fact that $ \widetilde{\Pi}_1 = \Pi_1$ since $\Sym^m f$ is self-dual.

The logarithm of $L(s, \Pi \times \widetilde{\Pi})$ has nonnegative Dirichlet coefficients for any isobaric representation $\Pi$ (\cite[Lemma a]{Hoffstein-Ramakrishnan}), so the same is true for the above auxiliary function. Moreover, since the analytic conductor is multiplicative, we can use the bounds from Lemmas \ref{lemma:log_conductor_sym_power_L} and \ref{lemma:conductor_L_mxm} to see that
\begin{align*}
    \log(C(\mathcal{A}_1)) &\leq 2\log|1+2i\gamma| + 4\log{C(\mathcal{L}_m(f, \gamma))} + \log{C(\mathcal{L}_{m \times m})} \\
    &\leq (m+5)^2\log\left(\Cr{zfr3}(k-1)Q(m+5)\sqrt{1+\gamma^2}\right) + 2\log{2} - \Cr{zfr4} - 4\Cr{zfr2}.
\end{align*} 
Applying Theorem \ref{general-zfr-theorem} with $a_0 = 3$, $a_1 = 1$, $b_0 = 4$, and $b_1 = 2$ yields \begin{align} \label{eq:zfr_1}
    1-\beta \geq \frac{2(2-\sqrt{3})^2}{\log(C(\mathcal{L})) + 2\Cr{zfr1}^{-1}\sqrt{3}(2-\sqrt{3})} \geq \frac{2(2-\sqrt{3})^2}{(m+5)^2\log\left(\Cr{zfr3}(k-1)Q(m+5)\sqrt{1+\gamma^2}\right)}.
\end{align}

Next, suppose $m \neq 2$ and $\gamma = 0$. In this case, we use the $L$-tuple $\mathcal{A}_2$ 
corresponding to 
\begin{align*}
    L(s, \Pi_2 \times \widetilde{\Pi}_2) &= \zeta(s)^3 L(s, \Sym^2{f})^3 L(s, \Sym^4{f}) \\
    &\quad \times
    L(s, \Sym^{m-2}{f})^2 L(s, \Sym^m f)^4 L(s, \Sym^{m+2}{f})^2\prod_{j=1}^m L(s, \Sym^{2j}{f}),
\end{align*} 
where $\Pi_2 = \mathbbm{1} \boxplus \Sym^2 f \boxplus \Sym^m f$. Again applying the bounds from Lemmas \ref{lemma:log_conductor_sym_power_L} and \ref{lemma:conductor_L_mxm}, we have \begin{align*}
     \log{C(\mathcal{A}_2)}
    &\leq (m+7)^2\log\left(\Cr{zfr3}(k-1)Q(m+7)\right) - (12\Cr{zfr2} + \Cr{zfr4} - 18\log{6} - 8\log{8}),
\end{align*} and so applying Theorem \ref{general-zfr-theorem} with $a_0 = 3, a_1 = 0, b_0 = 4, b_1 = 0$ gives \begin{align} \label{eq:zfr_2}
    1-\beta \geq \frac{2(2-\sqrt{3})^2}{(m+7)^2\log\left(\Cr{zfr3}(k-1)Q(m+7)\right)}.
\end{align}

Finally, for the case that $m = 2$ and $\gamma =0$, consider the $L$-tuple $\mathcal{A}_3$ 
corresponding to 
\begin{align*}
   L(s, \Pi_3 \times \widetilde{\Pi}_3) = \zeta(s)^2L(s, \Sym^2 f)^3 L(s, \Sym^4 f),
\end{align*} 
where $\Pi_3 = \mathbbm{1} \boxplus \Sym^2 f$. We have \begin{align*}
     \log(C(\mathcal{A}_3)) &\leq 
    14\log\left(9\Cr{zfr3}(k-1)Q\right) - (4\Cr{zfr2} + 14\log\left(18\Cr{zfr3}e\right) - 18\log{6} - 8\log{8}),
\end{align*} and so we obtain \begin{align} \label{eq:zfr_3}
    1-\beta \geq \frac{2(\sqrt{3}-\sqrt{2})^2}{14\log\left(9\Cr{zfr3}(k-1)Q\right)}.
\end{align} We arrive at Theorem \ref{theorem:zfr} by taking the maximum of the bounds in \eqref{eq:zfr_1}, \eqref{eq:zfr_2}, and \eqref{eq:zfr_3}.
\end{proof}

\begin{corollary}\label{minrho}
Let $\rho$ be a zero of $\Lambda(s, \Sym^m f)$ with $m \geq 1$, and suppose that $q_{\Sym^n f} \leq Q^{n+1}$ for all $1 \leq n \leq 2m+    1$. Then we have \begin{align*}
    |\rho| \geq \frac{\Cr{zfr5}}{(m+7)^2\log(\Cr{zfr3}(k-1)Q(m+7))}.
\end{align*}
\end{corollary}

\begin{proof}
Write $\rho = \beta + i \gamma$, and let \begin{align*}
    r_1(\gamma) := \frac{\Cr{zfr5}}{(m+7)^2\log(\Cr{zfr3}(k-1)Q(m+7)\sqrt{1+\gamma^2})}, \qquad r_2(\gamma) := \sqrt{r_1(0)^2 - \gamma^2}.
\end{align*}
   By Theorem \ref{theorem:zfr} together with the functional equation \eqref{eq:functional_equation} for $\Lambda(s, \Sym^m f)$, we have $\beta \geq r_1(\gamma)$. To show that $|\rho| \geq r_1(0),$ it suffices to show that any point $(r_2(\gamma), \gamma)$ on the circle with radius $r_1(0)$ satisfies $r_2(\gamma) \leq r_1(\gamma)$. For $|\gamma| \leq r_1(0)$, we observe that
 \begin{align*}
    r_1''(\gamma) &= \frac{\Cr{zfr5}\left(2\gamma^2+(\gamma^2-1)\log\left(\Cr{zfr3}(k-1)Q(m+7)\sqrt{1+\gamma^2}\right)\right)}{(m+7)^2(1+\gamma^2)^2\log\left(\Cr{zfr3}(k-1)Q(m+7)\sqrt{1+\gamma^2}\right)^3} \\
    &\geq -\frac{\Cr{zfr5}}{(m+7)^2(1+\gamma^2)^2\log\left(\Cr{zfr3}(k-1)Q(m+7)\sqrt{1+\gamma^2}\right)^2} \\
    &\geq -\frac{\Cr{zfr5}}{64\log\left(8\Cr{zfr3}\right)^2},
\end{align*} 
from which we find that \begin{align*}
    -r_2''(\gamma) \geq \frac{(m+7)^2\log(\Cr{zfr3}(k-1)Q(m+7))}{\Cr{zfr5}} \geq \frac{64\log(8\Cr{zfr3})}{\Cr{zfr5}} \geq \frac{\Cr{zfr5}}{64\log\left(8\Cr{zfr3}\right)^2} \geq -r_1''(\gamma).
\end{align*} 
Because $r_1(\gamma)$ and $r_2(\gamma)$ are even functions that agree at $\gamma = 0$, the fact that $r_2''(\gamma) \leq r_1''(\gamma)$ implies that $r_2(\gamma) \leq r_1(\gamma)$.
\end{proof}

\section{Vertical Distribution of Zeros}\label{verticalsection}

For $T \geq  1$, let $N(T, \Sym^m f)$ be the number of zeros $\rho = \beta + i \gamma$ of $L(s, \Sym^m f)$ with $0 < \beta < 1$ and $\lvert\gamma\rvert \leq T$. In this section, we will give an upper bound for $N(T,\Sym^m f)$ for $m\geq 1$. In particular, we prove the following theorem.

\begin{theorem} \label{Thm:RiemannVonMangoldt} 
Let $m\geq 1$ and $T \geq 1$, and suppose that $q_{\Sym^m f} \leq Q^{m+1}$ for some $Q \geq 1$. Then we have
\begin{align} \label{eq:RVM}
    N(T, \Sym^m f) &\leq \frac{1}{\pi} \bigg[ (m+1)T\log\left( \frac{\sqrt{2}}{2\pi e}(k-1)Q(m+2)T\right) + T\log(m+2) \nonumber\\
    &\quad \qquad  + \Cr{b_4}(m+1)\log(\Cr{b_{100}}(k-1)Q(m+2)T) + \frac{m+3+6\Cr{b_{130}}}{6T} \bigg].
\end{align}
where $\Cr{b_4} \leq 15.998$, $\Cr{b_{100}} \leq 17\,555$, and $\Cr{b_{130}} = 31.996$. Moreover, for $T \geq 200$, we have
\begin{equation} \label{eq:RVM_200}
    N(T, \Sym^m f) \leq \Cr{c_{303}} (m+1)  T \log ((k-1) Q(m+2)T),
\end{equation}
where $\Cr{c_{303}} := 0.593$, and we may write \begin{equation} \label{eq:RVM_1}
    N(1,\Sym^m f) \leq \Cr{c_{310}} (m+1) \log((k-1)Q(m+2)),
\end{equation}
where $\Cr{c_{310}} := 56.662$. 
\end{theorem}

For a fixed $1<\sigma_1 < 2$, let $R$ be the positively oriented rectangle with vertices $\sigma_1 \pm iT, 1-\sigma_1 \pm iT$. By the argument principle, we have
\begin{equation*}
    2\pi N(T, \Sym^m f) = \Delta_R \arg \Lambda(s, \Sym^m f).
\end{equation*} Furthermore, by \eqref{eq:functional_equation}, the contributions to $\Delta_R \arg \Lambda(s, \Sym^m f)$ from the left and right sides of the contour are equal. Accordingly, let $\mathcal{C}$ be 
the part of the contour from $\frac12 -iT$ to $\sigma_1 -iT$ to $\sigma_1 + iT$ to $\frac12 + iT$, so that
\begin{align} \label{eq:contour}
    \pi N(T, \Sym^m f) & = \Delta_{\mathcal{C}} \arg q_{\Sym^m f}^{s/2} + \Delta_{\mathcal{C}} \arg \gamma(s, \Sym^m f) + \Delta_{\mathcal{C}} \arg L(s, \Sym^m f).
\end{align}
In Lemmas \ref{lemma:gamma_contour} and \ref{lemma:L_contour}, we bound the contributions from the $\gamma$- and $L$-terms.

We will use the following explicit form of Stirling's formula from \cite{O}, as stated in \cite[Equation 2.4]{Tr}, several times throughout this section. 
\begin{lemma}[Stirling's Formula] \label{Stirling}
For $|\arg z| \leq \frac{\pi}{2}$, we have
   \begin{equation*} \log \Gamma(z) = \left(z-\frac12 \right)\log z - z + \frac12 \log 2\pi + \frac{\theta(z)}{6|z|},
    \end{equation*}
where $\theta(z) \in \mathbb{C}$ satisfies $|\theta(z)| \leq 1$.
\end{lemma}

We begin with the following bounds for $\Delta_{\mathcal{C}} \arg \Gamma_{\mathbb{C}}$ and $\Delta_{\mathcal{C}} \arg \Gamma_{\mathbb{R}}$.

\begin{lemma} \label{lemma:gammaC_and_gammaR_contour}
Let $T \geq 1$. For the curve $\mathcal{C}$ described at the start of the section, we have the following bounds on the summands of $\Delta_{\mathcal{C}} \arg \gamma(s, \Sym^m f)$:
\begin{align*}
    |\Delta_{\mathcal{C}} \arg \Gamma_{\mathbb{C}}(s + \mu)| &\leq -2T \log (2\pi e) + T \log \left(\left({\textstyle \frac12} + \mu \right)^2+ T^2\right) \nonumber \\
    & \ \ \ \ \  + 2\mu \tan^{-1}\left(\frac{T}{\frac12 + \mu}\right) + \frac{1}{3|\frac{1}{2}+\mu+ iT|}, \\
     \nonumber \\
    | \Delta_{\mathcal{C}} \arg  \Gamma_{\mathbb{R}}(s+r) | &\leq  -T \log 2 \pi e + T\log T + \frac{T}{2}\log \left(1 + \frac{(2r+1)^2}{4T^2}\right) \nonumber \\
    & \ \ \ \ \ + \frac{2r-1}{2} \tan^{-1} \left(\frac{2T}{2r + 1}\right) + \frac{2}{3|\frac12 + r + iT|}.
\end{align*}
\end{lemma}

\begin{proof} By Lemma \ref{Stirling}, we have
    \begin{align*}
        \Delta_{\mathcal{C}} \arg \Gamma(s + \mu) &= 2 \Im \log \Gamma \left(\tfrac{1}{2} + \mu + iT \right) \\
        &= 2 \mu \tan^{-1}\left(\frac{T}{\textstyle{\frac12 }+ \mu}\right) + T \log \left(\left(\frac12 + \mu\right)^2 + T^2\right) - 2T + \frac{\Im \theta(z)}{3|\frac12 + \mu + iT|},
    \end{align*}
where $|\theta| \leq 1$. Similarly, following the argument from Section 2 of \cite{Tr}, we see that
\begin{align*}
    \Delta_{\mathcal{C}} \arg \Gamma \left( \frac{s+r}{2} \right) 
    &= 2 \Bigg( \frac{T}{2} \log \frac{T}{2e} + \frac{T}{4} \log \left(1 + \frac{(2r+1)^2}{4T^2}\right) \\
    & \ \ \ \ \ \ \  \quad + \frac{2r-1}{4} \tan^{-1}\left(\frac{2T}{2r+1}\right) + \frac{\theta(z)}{3|\frac12 + r + iT|} \Bigg). 
\end{align*}
The stated bounds follow.
\end{proof}

\begin{lemma} \label{lemma:gamma_contour} 
Let $T \geq 1$. For the curve $\mathcal{C}$ described at the start of the section, we have
    \begin{align*}
        |\Delta_{\mathcal{C}} \gamma(s, \Sym^m f)| \leq  (m+1) T \log \left( \frac{\sqrt{2}}{2 \pi e} T (k-1)(m+2)  \right) + T \log (m+2) + \frac{m+3}{6} \, \frac{1}{T}.
    \end{align*}
\end{lemma}

\begin{proof}
Let $\delta_2(m) = 1$ for $m$ odd, and 0 otherwise. Then using Lemma \ref{lemma:gammaC_and_gammaR_contour}, we see that 
\begin{align} \label{eq:gamma_orig_expression}
    \left|\Delta_{\mathcal{C}} \arg \gamma(s, \Sym^m f)\right| &\leq -(m+ \delta_2(m)) T \log (2\pi e) + T S_1(m, T) + 2 S_2(m, T)  + \frac{1}{3} S_3(m,T)  \nonumber \\
    & \quad + (1-\delta_2(m)) \, |\Delta_{\mathcal{C}} \arg \Gamma_{\mathbb{R}} (s+r) |,
\end{align}
where we define the sums
\begin{align*}
    S_1(m, T) &:= \!\!\!\!\sum_{j=1}^{\frac12 (m+ \delta_2(m))} \log \bigg(
        T^2 + \bigg(
            \textstyle{\frac{1}{2}} + \Big(j-{\textstyle \frac{\delta_2(m)}{2}}\Big)(k-1)
        \bigg)^2
    \bigg), \\
    S_2(m, T) &:= \!\!\!\!\sum_{j=1}^{\frac12 (m+ \delta_2(m))} \left(j - {\textstyle \frac{\delta_2(m)}{2}}\right)(k-1) \tan^{-1} \Bigg( 
        \frac{T}{\frac{1}{2} + \left(j-\frac{\delta_2(m)}{2}\right)(k-1)}
    \Bigg), \\
    S_3(m, T) &:= \!\!\!\!\sum_{j=1}^{\frac12 (m+ \delta_2(m))} {\left|\textstyle{\frac12} + \left(j-{\textstyle \frac{\delta_2(m)}{2}}\right)(k-1) + iT \right|}^{-1}.
\end{align*} We proceed to estimate these sums. Noting that $\log(x^2 + y^2) \leq \log (2 x^2 y^2)$ for $x, y \geq 1$, we see
\begin{align*}
    S_1(m, T) &\leq \sum_{j=1}^{\frac12 (m+ \delta_2(m))} \log \bigg(
        2 T^2 \bigg(
            \frac{1}{2} + \Big(j-{\textstyle \frac{\delta_2(m)}{2}}\Big)(k-1)
        \bigg)^2
    \bigg) \\
    &\leq \frac12 (m + \delta_2(m)) \log (8T^2(k-1)^2) + 2 \sum_{j=1}^{\frac12 (m + \delta_2(m))} \log \left( 
         j - {\textstyle \frac{\delta_2(m)}{2}}
    \right) \\
    &\leq  (m + \delta_2(m)) \log \left( 
        \sqrt{2}{e^{-1}}\, T (k-1)
    \right) + (m+2) \log (m+2) \\
    &\quad - (2 - \delta_2(m)) \log (2 - \delta_2(m)),
\end{align*}
where the last step comes from observing that \[\sum_{j=1}^{\frac12 (m+ \delta_2(m))} \log \left( j - \frac{\delta_2(m)}{2} \right) \leq  \int_{1}^{\frac12 (m+ \delta_2(m)) + 1} \log \left(u - \frac{\delta_2(m)}{2}\right) \, du. \]
For the next two sums, we use the bounds
\[
    S_2(m,T) \leq \frac{m + \delta_2(m)}{2}\,T,  \qquad S_3(m, T) \leq \frac{m+ \delta_2(m)}{2} \, \frac{1}{T}.
\]
Finally, using Lemma \ref{lemma:gammaC_and_gammaR_contour} and observing that $\frac{2\sqrt{2}}{e}  > 1$ and $\log (1 + x) \leq x$ for $x \geq 1,$ we write 
\begin{align*}
    (1- \delta_2(m))& |\Delta_{\mathcal{C}} \arg \Gamma_{\mathbb{R}} (s+r)| \\
    &  \leq (1 - \delta_2(m)) \Bigg[
       -T \log (2\pi e) + T \log \left( 
            \frac{\sqrt{2}}{e} \, T(k-1)
       \right) + T \log 2 + \frac{9}{8T}  + \frac{T}{3} + \frac{2}{3T}
    \bigg].
\end{align*} The desired result follows by applying all of these bounds to \eqref{eq:gamma_orig_expression}.
\end{proof}

To bound $\Delta_{\mathcal{C}} \arg L(s, \Sym^m\! f),$ we will require bounds on $|\frac{\gamma'}{\gamma}(s)|$ and the number $N_1(T, \Sym^m f)$ of nontrivial zeros $\rho = \beta + i \gamma$ of $L(s, \Sym^m f)$ satisfying $|T-\gamma| \leq 1$. We will use the following explicit bound on $\left| \frac{\Gamma'}{\Gamma}(s) \right|$ given in the proof of Lemma 4 of \cite{KEN}.

\begin{lemma}[Ono--Soundararajan] \label{Ono-Sound}
For $s = \sigma + it$ with $\sigma \geq 1,$ we have
\[
    \left| \frac{\Gamma'}{\Gamma}(s) \right| \leq \frac{11}{3} + \frac{\log(1 + |s|^2)}{2}.
\]
\end{lemma}

\begin{lemma}\label{gammabound_symL} 
For $s = \sigma + iT,$ where $\sigma \geq \frac12$ and $|T| \geq 1$, we have the bound
 \begin{equation*}\left| \frac{\gamma'}{\gamma}(s, \Sym^m f) \right| \leq \frac{1}{2}(m+1)\log\left( |T| + \sigma + 1 + (k-1)(1 + {\textstyle\frac{m}{2}}) \right) + \Cr{b_1} m  + \Cr{b_2} + \frac{1}{|T|}, \end{equation*}
where $\Cr{b_1} := \frac{11}{6}+ \frac12 \log 2\pi \leq 2.7523$ and $\Cr{b_2} := \frac{41}{18} + \frac12 \log 2\pi\leq 3.1968$.
\end{lemma}

\begin{proof}
Again let $\delta_2(m)$ be 1 if $m$ is odd and $0$ otherwise. We may write 
\begin{align}   \label{eq:log_deriv_gamma_factor}
    \left| \frac{\gamma'}{\gamma}(s) \right| &\leq \left| -{\frac{1}{2}}(m+1)\log (2\pi) \right|
    + \left|\sum_{j=1}^{\frac{1}{2}(m + \delta_2(m))} \frac{\Gamma'}{\Gamma}\left(
        s + \left(j - {\textstyle \frac{\delta_2(m)}{2}}
    \right)(k-1)\right) \right| \nonumber\\
    & \quad + (1-\delta_2(m)) \left| \frac12 \, \frac{\Gamma'}{\Gamma}\left({ \frac{s+r}{2}}\right) \right| .
\end{align}
We now use Lemma \ref{Ono-Sound} to bound the second and third terms of \eqref{eq:log_deriv_gamma_factor}. First, we see
\begin{align} \label{eq:log_deriv_gamma_factor_sum_Gamma}
    \sum_{j=1}^{\frac{1}{2}(m + \delta_2(m))} & \left|  \frac{\Gamma'}{\Gamma} \left(
        s + \left(j - {\textstyle \frac{\delta_2(m)}{2}}
    \right)(k-1)\right) \right| \nonumber \\
    &\leq \frac{11}{6}(m+\delta_2(m)) + \frac{1}{2}\log \prod_{j=1}^{\frac{1}{2} (m + \delta_2(m))} \left(1 + \left|
        \sigma + iT + \left(j - {\textstyle \frac{\delta_2(m)}{2}} \right)(k-1)
    \right|^2 \right),
\end{align}
where the last term may be bounded by
\begin{align*}
    & \frac{m+\delta_2(m)}{2} \log (k-1) + \log \frac{\Gamma(A + \frac{m+\delta_2(m)}{2})}{\Gamma(A)},
\end{align*}
for $A := \frac{1+\sigma +|T|}{k-1}+1-\frac{\delta_2(m)}{2}$. By Lemma \ref{Stirling} and the bound $\log(1+x)\leq x$, we have

    \begin{align} \label{eq:log_deriv_gamma_factor_stirling}
         \left| \log \frac{\Gamma(A + \frac{m+\delta_2(m)}{2})}{\Gamma(A)} \right|
        & \leq A \log\left(\frac{A+\frac{m+\delta_2(m)}{2}}{A}\right)+\frac{m + \delta_2(m)-1}{2}\log\left(A +\frac{m+\delta_2(m)}{2}\right) \nonumber \\
        & \ \ \ \ \  +\frac12 \log A - \frac{m+\delta_2(m)}{2}+\frac{1}{6}\left( \left(A+\frac{m+\delta_2(m)}{2} \right)^{-1}+A^{-1}\right) \nonumber \\
        & \leq \frac{m + \delta_2(m)}{2}\log \left(A+\frac{m+\delta_2(m)}{2}\right) + \frac{4}{9}.
    \end{align}

We now bound the last term from \eqref{eq:log_deriv_gamma_factor}. Using the identity $s \Gamma(s) = \Gamma(s+1)$ together with Lemma \ref{Ono-Sound}, we find that
\begin{align} \label{eq:log_deriv_gamma_factor_r_contribution}
    \left|\frac12 \frac{\Gamma'}{\Gamma}\left(\frac{s+r}{2}\right)\right| & \leq  \frac12 \left| -\frac{2}{s+r} + \frac{\Gamma'}{\Gamma}\left( \frac{s+r}{2} + 1 \right) \right| \nonumber \\
    & \leq \frac{1}{|T|} + \frac{11}{6} + \frac12 \log \left( 1 + \sigma + |T| + (k-1)\left(1 + \frac{m}{2}\right) \right).
\end{align}

Using the bounds \eqref{eq:log_deriv_gamma_factor_sum_Gamma}, \eqref{eq:log_deriv_gamma_factor_stirling}, and \eqref{eq:log_deriv_gamma_factor_r_contribution} in \eqref{eq:log_deriv_gamma_factor} finishes the proof.
\end{proof}

\begin{lemma}\label{N1T} 
Let $|T| \geq 1$, and let $N_1(T, \Sym^m f)$ be the number of nontrivial zeros $\rho = \beta + i \gamma$ of $L(s,\Sym ^m f)$ satisfying $|T-\gamma| < 1$. If $Q$ satisfies $q_{\Sym^m f} \leq Q^{m+1}$, then for $m \geq 1$, we have
    \begin{align*}
        N_1(T, \Sym^m f) 
        &\leq \sqrt{5} \bigg(\frac{m+1}{2}\log \left(Q\left( |T| + \sigma_0 + 1 + (k-1) \left(1 + \frac{m}{2} \right)\right)  \right)+ \Cr{new1} m  + \Cr{new2} + \frac{1}{|T|}\bigg),
    \end{align*}
where $\Cr{new1} := \Cr{b_1} - \frac{\zeta'}{\zeta}(\sigma_0)\leq3.893 $, $\Cr{new2} :=\Cr{b_2} - \frac{\zeta'}{\zeta}(\sigma_0)\leq 4.337 $, and $\sigma_0 = \frac12(1+\sqrt{5}) \leq 1.61804$.
\end{lemma}

\begin{proof}
Let $s = \sigma + iT$, where $\sigma \geq 1$ and $|T| \geq 1$. To bound $N_1(T, \Sym^m f)$, 
we observe that $\Re\left( \frac{1}{s-\rho} \right) = \frac{\sigma-\beta}{(\sigma-\beta)^2 + (T-\gamma)^2}$ and use \eqref{eq:log_deriv_hadamard_prod} and \eqref{eq:B} to conclude 
\begin{align} \label{eq:inf_choose_sigma0}
    \left(\inf_{0<\beta<1} \frac{\sigma-\beta}{(\sigma-\beta)^2+1} \right) &N_1(T, \Sym^m f) \leq \sum_{|T-\gamma| < 1} \frac{\sigma-\beta}{(\sigma-\beta)^2+(T-\gamma)^2} \nonumber \\
    &= \Re \left(\frac{1}{2} \log q_{\Sym^m f} + \frac{L'}{L}(s, \mathrm{Sym}^m f) + \frac{\gamma'}{\gamma}(s, \mathrm{Sym}^m f)\right).
\end{align} To maximize $\inf_{0<\beta<1} \frac{\sigma-\beta}{(\sigma-\beta)^2+1}$, we fix $\sigma = \sigma_0 := \frac12(1+\sqrt{5})$. Bounding $|\frac{\gamma'}{\gamma}(s, \Sym^m f)|$ and $|\frac{L'}{L}(s, \Sym^m f)|$ using Lemma \ref{gammabound_symL} and equation \eqref{RTlemma} respectively in \eqref{eq:inf_choose_sigma0}, we obtain the stated bound.
\end{proof}

\begin{lemma} \label{lemma:L_contour} 
Let $T \geq 1$. For the curve $\mathcal{C}$ described at the start of the section, if $q_{\Sym^m f} \leq Q^{m+1}$ for some $Q \geq 1$, we have
\begin{align}
    \Delta_{\mathcal{C}} \arg L(s, \Sym^m f) \nonumber &\leq \Cr{b_4}(m+1) \log (\Cr{b_{100}}(k-1)(m+2)QT) + \frac{\Cr{b_{130}}}{T},
\end{align}
where $\Cr{b_4}:= \frac{9}{2} + \sqrt{5}(2+\pi) \leq 15.998,$ $\Cr{b_{100}} := (\sigma_0+1)\exp(\frac{\Cr{b_{19}}}{\Cr{b_4}}) \leq 17\,555$, and $\Cr{b_{130}} := 31.996$.
\end{lemma}

\begin{proof}
Following the example of \cite{D} and \cite{Tr}, we split the contour $\mathcal{C}   $ into three pieces: $\mathcal{C}_1$, $\mathcal{C}_2$, $\mathcal{C}_3$, corresponding to the line segments connecting $\frac{1}{2}-iT  $ to $\sigma_1-iT$, $\sigma_1-iT$ to $\sigma_1+iT$, and $\sigma_1+iT$ to $\frac{1}{2}+iT $, respectively.

We first bound the contribution from the line segment $\mathcal{C}_2$. Since $L$ is symmetric about the real axis, we have $\Delta _{\mathcal{C}_2}\arg L(s,\Sym^m f) = 2 \arg L(\sigma_1+iT, \Sym^m f)$. Using the fact that 
\[
    |\log L(\sigma_1 + iT, \Sym^m f)| \leq\left|\log\zeta(\sigma_1)^{m+1}\right| = (m+1)\log \zeta(\sigma_1)
\]
and writing 
\[
    \log L(\sigma_1+iT, \Sym^m f) = \log |L(\sigma_1+iT, \Sym^m f)|+i \arg L(\sigma_1+iT, \Sym^m f),
\]
we arrive at 
\begin{equation*}
    \Delta _{\mathcal{C}_2}\arg L(s,\Sym^m f) = 2\arg L(\sigma_1+iT, \Sym^m f) \leq  2(m+1)\log \zeta(\sigma_1).
\end{equation*}

We now consider the contributions from the horizontal segments. Again by the symmetry about the real axis, the contributions from the contours $\mathcal{C}_1$ and $\mathcal{C}_3$ are the same, so it suffices to estimate
\begin{equation*}
    \Delta_{\mathcal{C}_3} \arg L(s, \Sym^m f) = -\int_{\frac12 + iT}^{\sigma_1 + iT} \Im \left(\frac{L'}{L} (s, \Sym^m f)\right) ds.
\end{equation*}

Letting $s= \sigma+iT \in \mathcal{C}_3$ (that is, setting $\frac{1}{2} \leq \sigma \leq \sigma_1$), we use \eqref{eq:log_deriv_hadamard_prod} to compute $\frac{\Lambda'}{\Lambda}(s, \Sym^m f) - \frac{\Lambda'}{\Lambda}(\sigma_1 + iT),$ which produces the following equation:

\begin{align} \label{eq:L_contour_log_deriv_L_orig}
 \frac{L'}{L}(s, \Sym^m f)  
 & =  \frac{L'}{L}(\sigma_1+iT, \Sym^m f) + \left( \frac{\gamma'}{\gamma}(\sigma_1+iT, \Sym^m f) -\frac{\gamma'}{\gamma}(s, \Sym^m f) \right)   \nonumber \\ 
 & \ \ \ \ \ - \sum_{|T- \gamma|\geq 1} \frac{\sigma_1-\sigma}{(s-\rho)(\sigma_1+iT-\rho)}+ \sum_{|T- \gamma|\leq 1} \left( \frac{1}{\sigma_1+iT-\rho}-\frac{1}{s-\rho} \right). 
\end{align}

We proceed to bound each of these terms individually. 
We use \eqref{RTlemma} to bound the first term,
and the $\gamma$ contribution to \eqref{eq:L_contour_log_deriv_L_orig} may be bounded with Lemma \ref{gammabound_symL} via triangle inequality. The sum over nontrivial zeros proves more difficult. Using equation \ref{eq:inf_choose_sigma0}, 
as well as the bounds
\begin{equation*}
    \left|  \frac{\sigma_1-\sigma}{(s-\rho)(\sigma_1+iT-\rho)}\right| \leq \frac{\sigma_1 - \sigma}{(T-\gamma)^2}\leq \frac{\sigma_1 - \frac{1}{2}}{(T-\gamma)^2} 
\end{equation*}
and
\begin{equation*}
    \left(\frac{(T-\gamma)^2}{(\sigma_1-\beta)^2 + (T-\gamma)^2} \right) \left( \frac{\sigma_1-\beta}{\sigma_1-\frac{1}{2}}\right)  \geq \frac{1}{(\sigma_1-1)^2 + 1} \left(\frac{\sigma_1-1}{\sigma_1-\frac{1}{2}}\right),
\end{equation*}
we see that
    \begin{align} \label{eq:L_contour_small_distance_bound}
          \Bigg|\sum_{|T-\gamma|\geq 1}\frac{\sigma_1-\sigma}{(s-\rho)(\sigma_1+iT-\rho)} \Bigg| 
         &\leq  \frac{(\sigma_1-\frac{1}{2})((\sigma_1-1)^2+1)}{\sigma_1-1} \bigg( \Re \left(\frac{1}{2} \log q_{\Sym^m f} \right) \nonumber\\
       &  \quad +\Re\left( \frac{L'}{L}(\sigma_1+iT, \mathrm{Sym}^m f) + \frac{\gamma'}{\gamma}(\sigma_1+iT, \mathrm{Sym}^m f)\right) \bigg).
    \end{align}
We may again use Lemma \ref{gammabound_symL} and \eqref{RTlemma} to bound the last term in \eqref{eq:L_contour_small_distance_bound}. 
Finally, considering the local sum over zeros in  \eqref{eq:L_contour_log_deriv_L_orig}, we note that
\begin{equation} \label{eq:log_deriv_L_last_term}
    \Bigg| \sum_{|T-\gamma|\leq 1} \frac{1}{\sigma_1+iT-\rho}\Bigg|  \leq \sum_{|T- \gamma|\leq 1} \frac{1}{\sigma_1-1} = \frac{N_1(T, \Sym^m f)}{\sigma_1-1}. 
\end{equation}

Applying our work from Lemma \ref{gammabound_symL} and \eqref{RTlemma}, \eqref{eq:L_contour_small_distance_bound}, and \eqref{eq:log_deriv_L_last_term} in \eqref{eq:L_contour_log_deriv_L_orig} allows us to give explicit bounds on 
\begin{equation*}
    h(T) := -\frac{L'}{L}(s, \Sym^m f) - \sum_{|T-\gamma|\leq 1} \frac{1}{s-\rho}
\end{equation*}
in terms of $\sigma_1$. Now notice that for each $\rho$, we have
\begin{equation*}
    \int_{\frac{1}{2}+iT}^{\sigma_1+iT}\Im(s-\rho)^{-1} ds = \Delta_{\mathcal{C}_3} \Im(\log (s-\rho)) \leq \pi,
\end{equation*}
which gives the bound 
\begin{align} \label{eq:L_contour_C3_simple}
    \Delta_{\mathcal{C}_3} \arg L(s, \Sym^m f) 
    & =-     \int_{\frac{1}{2}+iT}^{\sigma_1+iT}\Im\left( \frac{L'}{L}(s+iT,\Sym^m f)\right)ds \nonumber \\
&   \leq \pi N_1(T,\Sym^m f)+ \left(\sigma_1-\frac{1}{2}\right) |h(T)|.
\end{align}
We will use Lemma \ref{N1T} in the above expression to bound $N_1(T, \Sym^m f)$. In particular, the main term of \eqref{eq:L_contour_C3_simple} will be
   \begin{equation*}
    \frac{m+1}{2}\left( \sqrt{5} \pi + \left(\sigma_1 - \frac12\right)\left(2 + g(\sigma_1) + \frac{\sqrt{5}}{\sigma_1 -1}\right) \right) \log \left( T + \max(\sigma_0, \sigma_1) + (k-1) \left( 1 + \frac{m}{2} \right) \right),
   \end{equation*} 
where $g(\sigma_1) = \frac{(\sigma_1-\frac12)((\sigma_1-1)^2+1)}{\sigma_1-1}$. We select $\sigma_1 = \frac32$ to minimize this term's coefficient, which becomes $\Cr{b_4} \cdot \frac{m+1}{2}$, where $\Cr{b_4} := \frac{9}{2} + (2 + \pi)\sqrt{5} < 15.998$.

With this value for $\sigma_1$, we can use \eqref{eq:L_contour_log_deriv_L_orig},
Lemma \ref{gammabound_symL}, \eqref{RTlemma}, \eqref{eq:L_contour_small_distance_bound}, and \eqref{eq:log_deriv_L_last_term}
to bound $h(T)$ explicitly; adding the contributions from $\mathcal{C}_1$ and $\mathcal{C}_2$, we find that
\begin{align*}
    \Delta_{\mathcal{C}} \arg L(s, \Sym^m f) & = 2 \Delta_{\mathcal{C}_3} \arg L(s, \Sym^m f) + \Delta_{\mathcal{C}_2} \arg L(s, \Sym^m f) \\
    &\leq \Cr{b_4}(m+1) \log\left( (\sigma_0+1)(k-1)\left(m+2\right)T \right)\\
    & \quad+ 2\Cr{b_9} \log q_{\Sym^m f} + \Cr{b_{14}} m + \Cr{b_{19}} + \frac{\Cr{b_{130}}}{T},
\end{align*}
where $\Cr{b_{130}}  := 31.996$, $\Cr{b_9} := 6.999$,  $\Cr{b_{14}} := 126.416$, and $\Cr{b_{19}} := 140.945$. Replacing $q_{\Sym^m f}$ with $Q^{m+1}$ and using the bound $(k-1)Q  \geq 11$, we arrive at the stated result. 
\end{proof}

\begin{proof}[Proof of Theorem \ref{Thm:RiemannVonMangoldt}]
The first bound \eqref{eq:RVM} follows from \eqref{eq:contour} after using the bounds from Lemmas \ref{lemma:gamma_contour} and \ref{lemma:L_contour}. To obtain \eqref{eq:RVM_200} and \eqref{eq:RVM_1}, we use the fact that $(k-1)Q \geq 11$. Indeed, we can bound the ratio of each summand in \eqref{eq:RVM} with the first term via a computer-assisted calculation, from which \eqref{eq:RVM_200} and \eqref{eq:RVM_1} follow.
\end{proof}

\section{Computing Residues of the Contour Integral} \label{Section:PerronIntegral}

In this section, we bound the integral $\frac{1}{y}\int_x^{x\pm y} \frac{L'}{L}(s,\Sym ^m f)\frac{x^{s+1}}{s(s+1)}$, which as shown in Section \ref{section-main-theorem} constitutes the main component of $\frac{1}{y}\int_x^{x\pm y} \Theta_m(u)du$ from Lemma \ref{erdos-turan}. In particular, we prove Lemmas \ref{asymptoticnontrivial}, \ref{trivialbound}, and \ref{residues}, which bound the contributions from the nontrivial zeros, the trivial zeros, and residues at $s=0,-1$ to the contour integral.

\subsection{Residues at Nontrivial Zeros}

We use the results from Sections \ref{horizontalsection} and \ref{verticalsection} to bound the contribution from the nontrivial zeros $\rho=\beta+i\gamma$ of $L(s,\Sym^m f)$, given by \begin{equation*}
    R_1(x, y, \Sym^m f) := \!\!\!\!\sum_{\text{nontrivial}\ \rho}\!\!\!\! \frac{(x+y)^{\rho+1} - x^{\rho+1}}{y\rho(\rho+1)}.
\end{equation*} This will constitute the largest contribution to the integral in Lemma \ref{integral}. 
\begin{lemma} \label{lemma:sum_over_rho_bounds} 
Let $x, y > 0$, and let $\rho = \beta + i\gamma \in \mathbb{C}$ with $0 < \beta < 1$. Then we have \begin{align*}
    \left|\frac{(x+y)^{\rho+1}-x^{\rho+1}}{y\rho(\rho+1)}\right| \leq \min\left\{\left(1 + \frac{y}{2x}|\rho|\right)\frac{x^{\beta}}{|\rho|}, \left(2\frac{x}{y} + 2 + \frac{y}{x}\right)\frac{x}{|\rho||\rho+1|}\right\}.
\end{align*}
\end{lemma}

\begin{proof}
Let $\theta = \frac{x}{y}$ and $\rho = \beta + i \gamma$. To prove the first bound, we write
\begin{align*}
    \frac{(x+y)^{\rho+1}-x^{\rho+1}}{y\rho(\rho+1)} 
    &= \frac{x^\rho}{\rho} + y x^{\rho-1}\, \frac{(1+ \theta)^{\rho+1} + 1- \theta (\rho+1)}{ \rho(\rho+1) \theta^2}
\end{align*}
and observe that 
\begin{equation} \label{eq:theta_bound}
    \bigg| \frac{(1+ \theta)^{\rho+1}-1 - (\rho+1)\theta}{\rho(\rho+1)\theta^2} \bigg| 
    = \bigg| \frac{1}{\theta^2} \int_1^{1+\theta} \frac{u^{\rho}-1}{\rho} du \bigg| \nonumber \leq \frac{1}{\theta^2} \int_0^\theta \frac{1}{|\rho|} |\rho| u \,du \nonumber = \frac12,
\end{equation}
where the inequality follows from the bound $(1 + u)^\rho - 1| \leq |\rho| u$. The second bound follows from \begin{equation*}
    \left| \frac{(x+y)^{\rho+1}-x^{\rho+1}}{y\rho(\rho+1)} \right| \leq \frac{x^{\beta + 1}}{|\rho||\rho+1|} \left( \left( 1 + \frac{y}{x}\right)^{\beta+1}+1 \right) \leq \frac{x}{|\rho||\rho+1|}(2\theta^{-1} + 2 + \theta). \qedhere
\end{equation*} 
\end{proof}

We will also need to estimate $\sum \frac{x^\rho}{\rho}$ and $\sum \frac{1}{\rho(\rho+1)}$, to bound the contribution from the low- and high-lying zeros, respectively. For convenience, write 
\begin{equation*}
    N(T, \Sym^m f) \leq G_1 T\log T + G_2 T + G_3 \log T + G_4 + \frac{G_5}{T},
\end{equation*}
where the $G_i$ are dependent only on $m,k,N$. Their precise values are derived from Theorem \ref{Thm:RiemannVonMangoldt}: 
\begin{align*}
    G_1 = \textstyle{\frac{1}{\pi}}(m+1),\quad   G_2= \textstyle{\frac{1}{\pi}}(m+1) \log \left(\textstyle{\frac{\sqrt{2}}{2\pi e}} \, (k-1) Q(m+2)\right) + \textstyle{\frac{1}{\pi}} \log (m+2),
    \end{align*}
    \begin{align*}
    G_3 = \textstyle{\frac{1}{\pi}} \Cr{b_4} (m+1),\quad    G_4 = \textstyle{\frac{1}{\pi}} \Cr{b_4} (m+1)  \log(\Cr{b_{100}} (k-1)Q(m+2)) , \quad 
    G_5 = \textstyle{\frac{1}{\pi}} \left(\Cr{b_{130}} + \textstyle{\frac{m+3}{6}}\right).
\end{align*}

\begin{lemma} \label{lemma:sum_inv_gamma} 
For $T \geq 200$ and $m \geq 1$, we have
\begin{equation*}
    \sum_{1< |\gamma| < T} \frac{1}{|\gamma|}  \leq \Cr{c_{233}} (m+1) \log^2((k-1)Q(m+2)T) 
\end{equation*}
for $\Cr{c_{233}} := 1.114$, where the sum runs across nontrivial zeros $\rho=\beta+i\gamma$ of $L(s,\Sym^m f)$. 
\end{lemma}
\begin{proof}
By partial summation, we have
\begin{align*}
  \sum_{1<|\gamma|<T}\frac{1}{|\gamma|}& \leq \frac{N(T, \Sym^m f)}{T} + \int_1^T \frac{N(t, \Sym^m f)}{t^2}dt    \\
  & \leq \frac{1}{2\pi}(m+1) \log^2 ((k-1) Q (m+2)T) \\
  & \quad + \left(G_2 + G_3  + G_4+ \tfrac12 G_5- \frac{1}{2\pi} (m+1) \log^2 ((k-1)Q(m+2))\right)  + \frac{G_5}{2T^2}. 
\end{align*} Noting that $(k-1)Q \geq 11$ and $T \geq 200,$ we can bound the ratio of each term with the first term via a computer-assisted calculation to conclude the stated result.
\end{proof}

\begin{lemma} \label{lemma:sum_inv_gamma2} 
For $T \geq 200$ and $m \geq 1$, we have
\begin{equation*}
    \sum_{|\gamma|>T} \frac{1}{\gamma^2}\leq  \Cr{c_{220}}(m+1)\frac{\log((k-1)Q(m+2)T)}{T}
\end{equation*}
for $\Cr{c_{220}} := 0.753$, where the sum runs across nontrivial zeros $\rho=\beta+i\gamma$ of $L(s,\Sym^m f)$. 
\end{lemma}

\begin{proof}
By partial summation, we have
    \begin{align*}
        \sum_{|\gamma|> T} \frac{1}{\gamma^2} &= -\frac{N(T, \Sym^m f)}{T^2} + 2 \int_T^\infty \frac{N(t, \Sym^m f)}{t^3} dt \\
        & \leq \frac{2}{\pi}(m+1) \frac{\log ( (k-1)  Q(m+2) T )}{T} + \frac{2}{\pi T} \left(
            (m+1) \log {\textstyle \frac{\sqrt{2}}{2\pi}}+ \log (m+2) 
        \right) \\
        & \ \ \ \ \ +  \frac{G_3\log T}{T^2} + \frac{G_3 + 2G_4}{2T^2}  + \frac{2 G_5}{3T^3}.
    \end{align*}
Again noting that $(k-1)Q \geq 11$ and $T \geq 200,$ we obtain the stated result.
\end{proof}

We now combine these results to estimate $R_1(x,y, \Sym^m f)$. 
\begin{proof}[Proof of Lemma \ref{asymptoticnontrivial}] Observe that $\sup_{|\gamma|\leq T }\Re(\rho) \leq 1-\eta_m (T)$ by Theorem \ref{theorem:zfr}.
Using Lemma \ref{lemma:sum_over_rho_bounds}, we may  write
\begin{align*}
        \Bigg|\sum_\rho \frac{(x+y)^{\rho+1} - x^{\rho+1}}{y \rho(\rho+1)}\Bigg|
     &\leq \sum_{|\gamma| \leq 1} \left|\frac{x^\rho}{\rho}\right| +  x^{1-\eta_m(T)} \sum_{|\gamma| \leq T} \frac{y}{2x} + x^{1-\eta_m(T)}  \sum_{1 < |\gamma| \leq T}\frac{1}{|\gamma|} \nonumber \\
     & \ \ \ \ \  + \left(2\frac{x}{y} + 2 + \frac{y}{x}\right)x \sum_{|\gamma| > T} \frac{1}{\gamma^2}. 
\end{align*}
In Lemmas \ref{lemma:sum_inv_gamma} and \ref{lemma:sum_inv_gamma2}, we have calculated the values of the sums $\sum_{1<|\gamma|<T}\frac{1}{|\gamma|}$ and  $\sum_{|\gamma|>T}\frac{1}{\gamma^2}$; after we use the bound on $N(T, \Sym^m f)$ from Theorem \ref{Thm:RiemannVonMangoldt} to bound $\sum_{|\gamma| \leq T} \frac{y}{2x}$ for $T \geq 200$, the only remaining piece is $\sum_{|\gamma|<1} \frac{x^\rho}{\rho}  $. For this, we note that 
\begin{equation*}
    \sum_{|\gamma|<1} \frac{x^\rho}{\rho} \leq N(1, \Sym^m f)\cdot \sup_{\rho}\left( \frac{1}{|\rho|}\right) \cdot x^{1-\eta_m(1)}. 
\end{equation*}
The quantity $N(1, \Sym^m f)$ is bounded in Theorem \ref{Thm:RiemannVonMangoldt}; additionally, Corollary \ref{minrho} gives us a lower bound for $\inf_\rho |\rho|$. 
Combining these, we arrive at the upper bound listed as the last summand in the lemma statement.
\end{proof}

\subsection{Residues at Trivial Zeros}

In Lemma \ref{trivialbound}, we bound the contribution to the integral in Lemma \ref{integral} given by the residues of all trivial zeros except those which may exist at $s=0$ and $s=-1$; the latter case is done separately, as the integrand $\frac{L'}{L}(s,\Sym^m f) \frac{x^{s+1}}{s(s+1)}$ has additional poles at these values. These residues are calculated separately in Lemma \ref{residues}.

\begin{proof}[Proof of Lemma \ref{trivial}]
As the trivial zeros correspond exactly with the poles of $\gamma(s, \Sym^m{f})$, the sum of the residues of $\frac{L'}{L}(s,\Sym^m f) \frac{x^{s+1}}{s(s+1)}$ over the trivial zeros $\rho \neq 0,-1$ is given by 
\begin{equation*}
    \sum_{j=1}^{\frac{m+1}{2}} \sum_{\ell = 0}^{\infty} \frac{x ^{-(j-\frac{1}{2})(k-1) - \ell+1}}{((j-\frac{1}{2})(k-1) + \ell ) \, ((j-\frac{1}{2})(k-1) + \ell -1)} 
\end{equation*}
in the case that $m$ is odd, and
\begin{equation*}
    \sum_{j=1}^{\frac{m}{2}} \sum_{\ell = 1}^{\infty} \frac{x ^{-j(k-1) - \ell+1}}{(j(k-1) + \ell )\,(j(k-1) + \ell -1)}  + 
    \sum_{j=1}^{m/2} \frac{\delta_{1,2}(j,k)x ^{-j(k-1) +1}}{j(k-1)  \,(j(k-1)  -1)}
\end{equation*}
in the case $m$ is even, where the $\delta_{1,2}(j,k)$ indicates that the term is to be omitted when $j=1$ and $k=2$; the case where $k=2$ and $m$ is even allows for an additional pole at $s=-1$, which is delegated to the calculation in Lemma \ref{residues}.

From here, it is straightforward to find an upper bound for each case, using the identity
\begin{equation*}
    \frac{d^2}{du^2} \left(\sum_{\ell=0}^{\infty} \frac{u^{-(c+\ell-1)}}{(c+\ell)(c+\ell-1)} \right) = \frac{u^{-c}}{u-1}
\end{equation*}
for $u > 1$ and $c >0$ and integrating. The case in which $k=2$ and $m$ is odd gives the maximum upper bound, namely
\begin{equation*}
\frac{4(m+1)}{3y}\left( \sqrt{x+y}- \sqrt{x}\right) \leq \frac{2(m+1)}{3\sqrt{x}}.\qedhere
\end{equation*}
\end{proof}
 
 \subsection{Additional Residues}
 Finally, we bound the contribution of $R_3(x,y)$, the residues of $\frac{L'}{L}(s, \Sym^m f) \frac{x^{s+1}}{s(s+1)}$ at $s=0,-1$.
 \begin{proof}[Proof of Lemma \ref{residues}]
 We first expand $\frac{x^{s+1}}{s(s+1)}$ about $s=0$ and $-1$, giving
 \begin{equation*}
     \frac{x^{s+1}}{s(s+1)} = \frac{x}{s} + x\log x - x + \cdots,\qquad      \frac{x^{s+1}}{s(s+1)} = -\frac{1}{s+1} -1-\log x  + \cdots,
 \end{equation*}
 respectively. By casework on the value of $m \pmod{4}$ and the size of $k$, we find that $L(s,\Sym ^m f)$ has zeros at $s=-1$ and $s=0$ of orders $2$ and $0$ for $4\nmid m, k=2$; orders $1$ and $1$ for $4 \mid m$, $k=2$; orders 1 and 0 for $4 \nmid m$, $k \geq 4$; and orders 0 and 1 for $4 \mid m, k \geq 4$, respectively. In each case, by expanding the series at each point and evaluating at $x+y$ and at $x$, we find
 \begin{align*}
 \Res\left(\frac{L'}{L}(s,\Sym ^m f) \frac{(x+y)^{s+1}-x^{s+1}}{s(s+1)}\right)  &  \leq |\lambda_0|y + \int_x^{x+y}\log(z)\,dz + \log\left(1+\frac{y}{x}\right)\\ & \leq |\lambda_0|y + y\log(x+y) + \frac{y}{x},
 \end{align*} 
where $\lambda_0$ is defined as either $\frac{L'(0)}{L(0)}$ or $\lim_{s \to 0}(\frac{L'(0)}{L(0)} - \frac{1}{s})$, whichever is well-defined. 
 
We proceed to bound $\lambda_0$. If $4 \nmid n$, then $\frac{\gamma'}{\gamma}$ extends analytically to $s=0$, and so we may write
 \begin{equation*}
 - \frac{L'}{L}(0,\Sym ^m f) = \frac{\log q_{\Sym^m{f}}}{2} + \frac{\gamma'}{\gamma}(0,\Sym ^m f) + \sum_{\rho}\frac{1}{\rho}.
 \end{equation*}
 To estimate $\sum_{\rho} \frac{1}{\rho}$, we first use Corollary \ref{minrho}, which gives us the bound
\begin{equation*}
     \frac{1}{\rho}+ \frac{1}{\overline{\rho}} \leq \frac{2}{\Cr{zfr5}}{(m+7)^2 \log(\Cr{zfr3}(k-1)Q(m+7))}.
\end{equation*} Hence we may bound the sum directly by \begin{equation*}
    \sum_\rho \frac{1}{\rho} \leq  \frac{N(1,\Sym^m f)}{\Cr{zfr5}}{(m+7)^2 \log(\Cr{zfr3}(k-1)Q(m+7))}+ \sum_{|\gamma|>1} \frac{1}{\rho}.
\end{equation*}
The last sum can be estimated using results from Section 5. In particular, we have
    \begin{align*}
         \sum_{|\gamma|\geq 1  }\frac{1}{\rho}
         & = \sum_{|\gamma | \geq 1} \left( \frac{\beta}{\sigma_1-\beta} \cdot \frac{(\sigma_1-\beta)^2 + \gamma^2}{\beta^2+\gamma^2} \right)\left(\frac{\sigma_1-\beta}{(\sigma_1-\beta)^2 + \gamma^2} \right),
      \intertext{from which choosing $\sigma_1 = \frac{3}{2},\  T=0$ yields}
         \sum_{|\gamma|\geq 1  }\frac{1}{\rho} &\leq  5 \bigg( \Re \left(\frac{1}{2} \log q_{\Sym^m f} \right) +\Re\left( \frac{L'}{L}\left(\frac{3}{2}, \mathrm{Sym}^m f\right) + \frac{\gamma'}{\gamma}\left(\frac{3}{2}, \mathrm{Sym}^m f\right)\right) \bigg)\\
         &\leq  5 \bigg(  \frac{m+1}{2} \log Q + (m+1) \left|\frac{\zeta'}{\zeta}\left( \frac{3}{2}\right) \right| + \frac{\gamma'}{\gamma}\left(\frac{3}{2}, \mathrm{Sym}^m f\right)\bigg).
    \end{align*} Thus it only remains to find explicit bounds for $\frac{\gamma'}{\gamma}\left( 0, \Sym^m f\right)$ and $\frac{\gamma'}{\gamma}\left( \frac{3}{2}, \Sym^m f\right)$. The first may be bounded via Lemma \ref{Ono-Sound}, as regardless of the value of $m \pmod{4}$, we may write
\begin{align*}
    \frac{\gamma'}{\gamma}\left( \frac{3}{2}, \Sym^m f\right) 
    & \leq \frac{11(m+1)}{6}+ \frac{m+1}{4} \log \left( 1 + \left( \frac{3}{2}+\frac{(m+1)(k-1)}{2}\right)^2\right) -\frac{(m+1)\log(2\pi)}{2}\\
    & \leq \Cr{res32}(m+1) + \frac{m+1}{2} \log\left( (m+1)(k-1)\right),
\end{align*}
where $\Cr{res32} = \frac{11}{6} +\frac{1}{2} \log(\frac{7}{4}) - \frac{1}{2}\log(2\pi) \leq 1.1943$.

As for $\frac{\gamma'}{\gamma}(0, \Sym^m f)$, we note that $(j-1/2)(k-1)\geq1$ in all cases except $j=1,k=2$, so that we may use Lemma \ref{Ono-Sound} again. As $\Gamma'(\frac{1}{2})/\Gamma(\frac{1}{2}) <-1.9   $, we omit it by non-negativity and reach the bound
\begin{align*}
    \frac{\gamma'}{\gamma}\left( 0, \Sym^m f\right) 
    & \leq \Cr{res32}(m+1)  +  \frac{m+1}{2} \log\left( (m+1)(k-1)\right).
\end{align*}
In the case that $\frac{L'}{L}(s,\Sym^m{f})$ does not extend analytically to 0, the argument must be changed slightly. Indeed, as $L(s, \Sym^m{f})$, we may write
\begin{equation*}
 \Re\left( \sum_\rho \frac{1}{\sigma-\rho}\right) = \frac{\log q_{\Sym^m\!{f}}}{2}+ \left(\frac{\gamma'}{\gamma}(\sigma, \mathrm{Sym}^m f) + \frac{1}{s}\right) + \left(\frac{L'}{L}(\sigma, \mathrm{Sym}^m f)-\frac{1}{s}\right),
\end{equation*} where the second and third terms extend analytically to $\sigma=0$.
Thus, taking $\frac{L'}{L}(s,\Sym^m{f}) - \frac{1}{s}$ as $s\rightarrow 0$, we see that it is indeed bounded by the same quantity, as the contributions from the poles of $\frac{L'}{L}$ and $\frac{\gamma'}{\gamma}$ cancel. Hence in either case we may use the above bound, so that 
\begin{align*}
    \lambda_0 
    & \leq \frac{N(1, \Sym^m f)}{\Cr{zfr5}}{(m+7)^2 \log(\Cr{zfr3}(k-1)Q(m+7))} \\
    &\quad + 3(m+1) \log((k-1)Q(m+1))+ \Cr{reslast}(m+1),
\end{align*}
where $\Cr{reslast}  := 6 \Cr{res32} + |\zeta'(\frac{3}{2})/\zeta(\frac{3}{2})| < 8.6705$. 
\end{proof}

\section{Application to the Atkin--Serre Conjecture} \label{section-atkin-serre}

Here we apply Theorem \ref{main} to the Atkin--Serre conjecture. We require the following lemma.

\begin{lemma}\label{ASfirstlemma}
Let $f$ be as in Theorem \ref{AS}. Then for $x \geq 3$, 
  \begin{equation*}
 \# \{ x < p \leq 2x : p \nmid N,\ \cos \theta_p \in I \} \leq  \Cr{c_{101}}(\pi(2x) - \pi(x)) \left(  \mu_{\ST}(I) + \Cr{c_{100}} \frac{\log((k-1)N \log{x})}{\sqrt{\log{x}}}\right), 
  \end{equation*}
  where $\Cr{c_{101}} := 3.015$.
\end{lemma}
\begin{proof}
Applying the triangle inequality to $\pi_{f,I}(2x) - \pi_{f,I}(x)$ and using either Theorem \ref{main2} or Theorem \ref{mainEC}, we first arrive at the result 
\begin{equation*}
\pi_{f,I}(2x) - \pi_{f,I}(x) \leq (\pi(2x)+ \pi(x))\left(  \mu_{\ST}(I) + \Cr{c_{100}} \frac{\log((k-1)N \log{x})}{\sqrt{\log{x}}}\right).
\end{equation*}
Now, we use that $\pi(2x) + \pi(x) \sim 3(\pi(2x) - \pi(x)) \sim 3\pi(x)$. In particular, using the bound $ \frac{x}{\log{x}-1.1}<\pi(x) < \frac{x}{\log{x}-1}$ for $x>60\, 184$ given by \cite{Du}, we see that
\begin{align}
    \pi(2x) + \pi(x) & \leq \frac{2x}{\log(2x)-1.1} + \frac{x}{\log{x}-1.1}\nonumber\\
    &\leq \Cr{c_{101}}\left( \frac{2x}{\log(2x)-1} - \frac{x}{\log{x} - 1.1}\right)\nonumber\\
    & \leq \Cr{c_{101}} \left( \pi(2x) - \pi(x)\right), \nonumber
\end{align}
 where the bound holds for $x \geq 10^{100}$ with $\Cr{c_{101}}$ defined as in the statement of the lemma. Noting that the listed bound holds trivially for $3 \leq x \leq 10^{100}$, we arrive at the desired result.
 \end{proof}

\begin{proof}[Proof of Theorem \ref{AS}]
First, note that as $a_f(p) = 2p^{\frac{k-1}{2}}\cos \theta_p$ and since $\ell(x) := \frac{\log((k-1)N \log{x})}{\sqrt{\log{x}}}$ is decreasing in $x$, we see 
\begin{multline*}
    \# \left\{ x < p \leq 2x : |a_f(p)| \leq  2p^{\frac{k-1}{2}} \tfrac{\log ((k-1)N \log p)}{\sqrt{\log p}} \right\} \\ \leq \# \left\{ x<p\leq 2x : |\cos \theta_p| \leq \tfrac{\log((k-1)N \log{x})}{\sqrt{\log{x}}}\right\}.
\end{multline*} The statement is trivially true whenever $\ell(x) > 1$, so we may assume $\ell(x) \leq 1$. Let $I = \left[ \frac{\pi}{2} - \ell(x),\ \frac{\pi}{2}+ \ell(x)\right]$, so that if $\cos \theta_p \in [-\ell(x),\ell(x)]$, then $\theta_p \in I$. Using the Taylor expansion for $\sin^2{\theta}$, we may write

\begin{align*}
   \mu_{\ST}(I) & := \frac{2}{\pi}\int_{\frac{\pi}{2}-\ell(x)} ^{\frac{\pi}{2} + \ell(x)} (\sin^2 {\theta}) \,d\theta = \frac{4}{\pi}\left( \ell(x) - \frac{\ell(x)^3}{3} + \frac{\ell(x)^5}{15} - \cdots  \right),
\end{align*}
 which implies $\mu_{\ST}(I)     \leq  \frac{4}{\pi }\ell(x)$. Hence by Lemma \ref{ASfirstlemma}, the stated result holds for the constant $\Cr{c_{101}}\left(\Cr{c_{100}} + \frac{4}{\pi}\right)\leq 179$.
\end{proof}

\appendix

\section{Conductor of Symmetric Powers of Elliptic Curves} \label{ECappendix}

This section presents a short proof of the upper bound on the arithmetic conductor $q_{\Sym^m f}$ of $L(s, \Sym^m f)$, for $m \geq 1$ stated in Section \ref{section-symmetric-powers}.

\begin{theorem} \label{ECbound}
Let $E/\mathbb{Q}$ be a non-CM elliptic curve with conductor $N=q_{\Sym^1 f}$, and suppose $m \geq 1$. Then we have \begin{align*}
    q_{\Sym^m f} \leq N^{m+1}.
\end{align*}
\end{theorem}

\noindent
Our bound will rely crucially on the computations of \cite{MW}, and our result is a uniform sharpening of the bound $q_{\Sym^m f} \ll N^{6m}$ from \cite[Appendix A.2]{David}. We write 
\begin{align*}
    q_{\Sym^m f} = \prod_p p^{\epsilon_m(p) + \delta_m(p)},
\end{align*}
where $\epsilon_m(p)$ and $\delta_m(p)$ are the local tame conductor and local wild conductor at $p$ respectively.  
\begin{proof}
We freely use results from \cite[Section 3]{MW}. If $p$ is a prime of multiplicative reduction, then $\epsilon_m(p)=m$ and $\delta_m(p)=0$. If $p$ is of additive reduction, we know that $\epsilon_m(p)\leq m+1$; since $\delta_m(p)=0$ for $p\geq 5$, it suffices to bound the wild conductors at $2$ and $3$. Recall that in the case of additive reduction at $p$, we have $v_p(N) \geq 2$, where $v_p$ denotes the $p$-adic valuation. Let \begin{align*}
    d_2 := \max\left\{0,\frac{v_2(N)}{2} - 1\right\}(m+1) \qquad \text{and} \qquad d_3 := \max\left\{0, \frac{v_3(N)}{2} - 1\right\}(m+1).
\end{align*} By \cite[Table 2]{MW} and \cite[Table 3]{MW}, we see that $\delta_m(2)\leq d_2$ and $\delta_m(3) \leq d_3$, respectively. Thus, we may write
\begin{align*}
q_{\Sym^m f} &\leq 2^{d_2} 3^{d_3} \prod_{p\mid N} p^{m+1}\leq N^{m+1}. \qedhere
\end{align*}
\end{proof}

\section{Table of Constants}\label{appendixconstants}

The following constants are rounded appropriately (up or down) based on their usage in context. Explicit formulas, when available, are given in the first introduction of the constant in the text.
\begin{multicols}{4}
\begin{itemize}
    \item[] $\Cl[abcon]{c_{100}}$: 58.1
    \item[] $\Cl[abcon]{c_{99}}$: 58.084
    
    \item[] $\Cl[abcon]{c_{60}}$: 1.2323

    \item[] $\Cl[abcon]{zfr1}$: 1.721 
    \item[] $\Cl[abcon]{zfr2}$: 3.158 
    \item[] $\Cl[abcon]{zfr3}$: 1.214 
    \item[] $\Cl[abcon]{zfr4}$: 11.63 
    \item[] $\Cl[abcon]{zfr5}$: 0.1435

    \item[] $\Cl[abcon]{b_1}$: 2.753
    \item[] $\Cl[abcon]{b_2}$: 3.1968 
    
    \item[] $\Cl[abcon]{new1}$: 3.893 
    \item[] $\Cl[abcon]{new2}$: 4.337

    \item[] $\Cl[abcon]{b_4}$: 15.998 
        \item[] $\Cl[abcon]{b_{100}}$: 17\,555 
    \item[] $\Cl[abcon]{b_{130}}$: 31.996
 
    \item[] $\Cl[abcon]{b_9}$: 6.999
    \item[] $\Cl[abcon]{b_{14}}$: 126.725
    \item[] $\Cl[abcon]{b_{19}}$: 140.945

    \item[] $\Cl[abcon]{c_{101}}$: 3.015 

    \item[] $\Cl[abcon]{c_{233}}$: 1.114
    \item[] $\Cl[abcon]{c_{220}}$: 0.753
    \item[] $\Cl[abcon]{c_{303}}$: 0.593
    \item[] $\Cl[abcon]{c_{310}}$: 56.662
    \item[] $\Cl[abcon]{res32}$: 1.1943 
    \item[] $\Cl[abcon]{reslast}$: 8.6705 
\end{itemize}
\end{multicols}

\end{document}